\documentclass[preprint]{elsarticle}

\usepackage{float}
\usepackage{amsmath,amsbsy,amsthm}
\usepackage{amstext,amsfonts,amsmath,amssymb}

\usepackage{atbegshi}
\AtBeginDocument{\AtBeginShipoutNext{\AtBeginShipoutDiscard}}

\makeatletter
\def\@author#1{\g@addto@macro\elsauthors{\normalsize%
\def\baselinestretch{1}%
\upshape\authorsep#1\unskip\textsuperscript{%
\ifx\@fnmark\@empty\else\unskip\sep\@fnmark\let\sep=,\fi
\ifx\@corref\@empty\else\unskip\sep\@corref\let\sep=,\fi}%
\def\authorsep{\unskip,\space}%
\global\let\@fnmark\@empty
\global\let\@corref\@empty  
\global\let\sep\@empty}%
\@eadauthor={#1}
}
\makeatother

\newtheorem{theorem}{Theorem}[section]

\newtheorem{lemma}[theorem]{Lemma}
\newtheorem{remark}[theorem]{Remark}

\newfont{\eurorm}{eurm10 scaled 1000}
\newfont{\eurosm}{eurm10 scaled 700}
\newfont{\euroft}{eurm10 scaled 500}

\def\nbc#1#2{{\Big (}\!\!\mbox{$\begin{array}{c}\mbox{\footnotesize $#1$} \\[-0.9 mm]\mbox{\tiny $#2$}\end{array}$}\!\!{\Big )}}
\def\NBC#1#2{{\Big (}\!\!\mbox{$\begin{array}{c}\mbox{\footnotesize $#1$} \\[-0.11 mm]\mbox{\footnotesize $#2$}\end{array}$}\!\!{\Big )}}

\def\eu#1{\mbox{\eurorm #1}}
\def\euft#1{\mbox{\euroft #1}}
\def\eusm#1{\mbox{\eurosm #1}}

\newcommand{\myplus}{\ensuremath{+}}
\newcommand{\myast}{\ensuremath{\ast}}

\journal{arXiv.org}

\begin{document}

\begin{frontmatter}

\date{October 23, 2015.}

\title{\bf \large A METHOD OF PROVING A CLASS OF INEQUALITIES OF MIXED TRIGONOMETRIC POLYNOMIAL FUNCTIONS}

\author{\bf Branko Male\v sevi\' c\corref{cor}}
\ead{branko.malesevic@etf.rs}
\cortext[cor]{Corresponding author, {\it Telephone:} +381113218321, {\it Fax:} +381113248681 \\
Research is partially supported by Serbian Ministry of Education,
Science and Technological Development, Projects ON 174032 and III 44006.}
\author{\bf Milica Makragi\' c},
\ead{milica.makragic@etf.rs}

\address{Faculty of Electrical Engineering, University of  Belgrade,   \\
Bulevar kralja Aleksandra 73, 11000 Belgrade, Serbia}

\begin{abstract}
In this article we consider a method of proving a class of inequalities of the form (\ref{Ineq_1}).
The method is based on the precise approximations of the sine and cosine functions
by Maclaurin polynomials of given order. By using this method we present new proofs
of some inequalities of C.$-$P. Chen, W.$-$S. Cheung [J. Inequal. Appl.~2012:72~(2012)]
and Z.$-$J. Sun, L. Zhu [ISRN~Math.~Anal.~(2011)].
\end{abstract}


\end{frontmatter}

\section{Introduction} 

\smallskip
\noindent
\qquad
In this article we consider a method of proving trigonometric inequalities of the form
\begin{equation}
\label{Ineq_1}
f(x)=\sum\limits_{i=1}^{n}\alpha_{i}x^{p_{i}}\!\cos^{q_{i}}\!\!x\sin^{r_{i}}\!\!x>0
\end{equation}
for $x \!\in\! (\delta_{1},\delta_{2})$, $\delta_1 \!\leq\! 0 \!\leq\! \delta_2$ and $\delta_{1} \!<\! \delta_{2}$;
where $\alpha_{i} \!\in\! \mathbb{R} \!\setminus\! \lbrace0\rbrace$, \mbox{$p_{i}, q_{i}, r_{i}
\!\in\! \mathbb{N}_{0}$}~and~\mbox{$n \!\in\! \mathbb{N}$}. The function $f(x)$ is
a mixed trigonometric polynomial function, see~\cite{Dong_Yu_Yu_2013}.
These functions appear in~the theory of analytic inequalities \cite{Aharonov_Elias_2013}-\cite{Mortici_Debnath_Zhu_2015},
\cite{Guo_Luo_Qi_2013b}-\cite{Jiang_Qi_2012}, \cite{Klen_Visuri_Vuorinen_2010},
\cite{Malesevic_Banjac_Jovovic_2015}-\cite{Zhu_2006b}.

\noindent
\qquad
In the article \cite{Mortici_2011} a natural approach of proving some concrete \mbox{examples} of inequalities of the form (\ref{Ineq_1}) has been shown. This method is based on the direct comparison
of the sine and cosine functions with the corresponding Maclaurin polynomials. However,
the above-mentioned method is not appliable to the function $\cos^2\!x$ in the whole interval
$[0,\frac{\pi}{2}]$ and to the function $\sin^2\!x$ in the whole interval $[0,\pi]$.
Based on that fact, it is not advisable to make comparisons of $\cos^{q_{i}}\!\!x\!\cdot\!\sin^{r_{i}}\!\!x$
with the product of the corresponding Maclaurin approximations of the cosine and sine functions
raised to the powers $q_{i}$ and $r_{i}$ respectively. Therefore, one of the possibilities
is to make a transformation of $\cos^{q_{i}}\!\!x\!\cdot\!\sin^{r_{i}}\!\!x$ into
the sum of sines and cosines of multiple angles. In the continuation of the article,
we have explained a method of proving inequalities of the form (\ref{Ineq_1})
by transforming the function $f(x)$ into an equivalent form in which sines and cosines of multiple angles appear.

\noindent
\qquad
Let $\varphi : [a,b] \longrightarrow \mathbb{R}$
be a function which is differentiable on a segment $[a,b]$ and differentiable arbitrary
number of times on a right neighbourhood of the point $x=a$ and denote by $T^{\varphi, a}_{m}(x)$
the Taylor polynomial of the function $\varphi(x)$ in the point $x=a$ of the order $m$. If~there
is some $\eta\!>\!0$ such that holds:
$
T^{\varphi, a}_{m}(x) \geq \varphi(x),
$
for $x \!\in\! (a,a+\eta) \!\subset\! [a,b]$; then let us define
$\overline{T}^{\,\varphi,a}_{m}(x)=T^{\,\varphi, a}_{m}(x)$ and $\overline{T}^{\,\varphi,a}_{m}(x)$
present an upward approximation of the function $\varphi(x)$ on right neighbourhood $(a,a+\eta)$ of
the point $a$ of the oreder~$m$. Analogously, if there is some $\eta>0$ such that holds:
$
T^{\varphi, a}_{m}(x) \leq \varphi(x),
$
for $x \!\in\! (a,a+\eta) \!\subset\! [a,b]$; then let us define
$\underline{T}^{\varphi,a}_{\,m}(x)=T^{\,\varphi, a}_{m}(x)$ and $\underline{T}^{\varphi,a}_{\,m}(x)$
present a downward approximation of the function $\varphi(x)$ on right neighbourhood $(a,a+\eta)$ of
the point $a$ of the order $m$. Let us note that it is possible to analogously define upward and downward
approximations on some left neighbourhood of a point.  Further on, we observe the function $\varphi(x)$ as
a function from the set $\lbrace \sin x, \cos x \rbrace$.

\noindent
\qquad
Observing Maclaurin approximations of the sine and cosine functions, we notice that
$\overline{T}^{\,\sin,0}_{1}(x)$ is above and $\underline{T}^{\sin,0}_{\,3}(x)$ is below
the graph of the function $\sin x$ for $x\!>\!0$ and $\underline{T}^{\sin,0}_{\,1}(x)$
is below and $\overline{T}^{\,\sin,0}_{3}(x)$ is above the graph of the function $\sin x$
for $x\!<\!0$ as well as that $\overline{T}^{\,\cos,0}_{0}(x)$ is above and $\underline{T}^{\cos,0}_{\,2}(x)$
is below the graph of the function $\cos x$. The previous facts are stated precisely and generalized through
the following statements:

\begin{lemma}
{\rm \textbf{\!(i)}} For the polynomial $T_n(t)=\!\!\!\!\displaystyle\sum\limits_{i=0}^{(n-1)/2}
\!\!\dfrac{(-1)^it^{2i+1}}{(2i+1)!}$, where $n=4k+1$ $k\in \mathbb{N}_{0}$, it is valid:
\begin{equation}
\label{Ineq_2}
\Big{(}\forall t \in \big{[}0,\mbox{\small $\sqrt{(n+3)(n+4)}$}\,\big{]}\Big{)}\,\overline{T}_n(t)\geq \overline{T}_{n+4}(t)\geq \sin t,
\end{equation}

\vspace*{-2.5 mm}

\begin{equation}
\label{Ineq_3}
\Big{(}\forall t \in \big{[}\mbox{\small $-\sqrt{(n+3)(n+4)}$},0\big{]}\Big{)}\,\underline{T}_n(t)
\leq \underline{T}_{n+4}(t)\leq \sin t.
\end{equation}
For the value $t=0$ the inequalities in $(\ref{Ineq_2})$ and $(\ref{Ineq_3})$ turn into equalities.
For the values $t\!=\!\mbox{\small $\pm\sqrt{(n+3)(n+4)}$}$ the equalities $\overline{T}_n(t)\!=\!\overline{T}_{n+4}(t)$ and $\underline{T}_n(t)\!=\!\underline{T}_{n+4}(t)$ are true, respectively.

\smallskip
\noindent
{\rm \textbf{(ii)}}
For the polynomial $T_n(t)=\!\!\!\displaystyle\sum\limits_{i=0}^{(n-1)/2} \!\!\dfrac{(-1)^it^{2i+1}}{(2i+1)!}$,
where $n=4k+3$, $k\in \mathbb{N}_{0}$, it is valid:
\begin{equation}
\label{Ineq_4}
\Big{(}\forall t \in \big{[}0,\mbox{\small $\sqrt{(n+3)(n+4)}$}\,\big{]}\Big{)}\,\underline{T}_n(t)\leq \underline{T}_{n+4}(t)\leq \sin t,
\end{equation}

\vspace*{-2.5 mm}

\begin{equation}
\label{Ineq_5}
\Big{(}\forall t \in \big{[}\mbox{\small $-\sqrt{(n+3)(n+4)}$},0\big{]}\Big{)}\,\overline{T}_n(t)\geq \overline{T}_{n+4}(t)\geq \sin t.
\end{equation}

\break

\noindent
For the value $t=0$ the inequalities in $(\ref{Ineq_4})$ and $(\ref{Ineq_5})$ turn into equalities.
For the values $t=\mbox{\small $\pm\sqrt{(n+3)(n+4)}$}$ the equalities $\underline{T}_n(t)=\underline{T}_{n+4}(t)$ and $\overline{T}_n(t)=\overline{T}_{n+4}(t)$ are true, respectively.
\end{lemma}
\begin{proof}
{\rm \textbf{(i)}}
Let us suppose that $ 0<t\leq\sqrt{(n+3)(n+4)}$, then it is correct:
$$
\begin{array}{l}
\overline{T}_n(t)=\displaystyle\sum\limits_{i=0}^{2k}
\dfrac{(-1)^it^{2i+1}}{(2i+1)!}>
\displaystyle\sum\limits_{i=0}^{\infty}\dfrac{(-1)^it^{2i+1}}{(2i+1)!}=\sin t
                                                                                      \\[3.5 ex]
\Longleftrightarrow\;
\displaystyle\sum\limits_{j=1}^\infty
\dfrac{t^{4(k+j)-1}}{{\big(}4(k+j)-1{\big)}!}
\underbrace{\left(1-\dfrac{t^2}{{\big(}4(k+j){\big)}{\big(}4(k+j)+1{\big)}}\right)}_{(\geq0)}
>0;
\end{array}
$$
then,

\smallskip
$$
\overline{T}_{n+4}(t)=\overline{T}_n(t)-\dfrac{t^{n+2}}{(n+2)!}
\underbrace{\left(1-\dfrac{t^2}{(n+3)(n+4)}\right)}_{(\geq0)}\leq \overline{T}_n(t)
$$
and
$$
\begin{array}{l}
\overline{T}_{n+4}(t)=\displaystyle\sum\limits_{i=0}^{2k+2}
\dfrac{(-1)^it^{2i+1}}{(2i+1)!}>\displaystyle\sum\limits_{i=0}^{\infty}
\dfrac{(-1)^it^{2i+1}}{(2i+1)!}=\sin t
                                                                                      \\[3.5 ex]
\Longleftrightarrow\;
\displaystyle\sum\limits_{j=1}^\infty
\dfrac{t^{4(k+j)+3}}{{\big(}4(k+j)+3{\big)}!}
\underbrace{\left(1-\dfrac{t^2}{{\big(}4(k+j)+4{\big)}{\big(}4(k+j)+5{\big)}}\right)}_{(>0)}
>0.
\end{array}
$$
The corresponding equalities at the endpoints of the segment $\big{[}0,
\mbox{\small $\sqrt{(n+3)(n+4)}$}\,\big{]}$ are easily checked.
Overall, (\ref{Ineq_2}) has been proved. For
$t \in \big{[}\mbox{\small $-\sqrt{(n+3)(n+4)}$}, 0\big{)}$,
(\ref{Ineq_3}) is valid on the basis of the odd property of the function $\sin x$.
Overall, (\ref{Ineq_3}) has been proved.

\smallskip
\noindent
{\rm \textbf{(ii)}}
Let us suppose that $ 0<t\leq\sqrt{(n+3)(n+4)}$, then it is correct:
$$
\begin{array}{l}
\underline{T}_n(t)=\displaystyle\sum\limits_{i=0}^{2k+1}
\dfrac{(-1)^it^{2i+1}}{(2i+1)!}<
\displaystyle\sum\limits_{i=0}^{\infty}\dfrac{(-1)^it^{2i+1}}{(2i+1)!}=\sin t
                                                                                           \\[3.5 ex]
\Longleftrightarrow\;
\displaystyle\sum\limits_{j=1}^\infty
-\dfrac{t^{4(k+j)+1}}{{\big(}4(k+j)+1{\big)}!}
\underbrace{\left(1-\dfrac{t^2}{{\big(}4(k+j)+2{\big)}{\big(}4(k+j)+3{\big)}}\right)}_{(\geq0)}
<0;
\end{array}
$$
then,

\smallskip
$$
\underline{T}_{n+4}(t)=\underline{T}_n(t)+\dfrac{t^{n+2}}{(n+2)!}
\underbrace{\left(1-\dfrac{t^2}{(n+3)(n+4)}\right)}_{(\geq0)}\geq \underline{T}_n(t)
$$
and
$$
\begin{array}{l}
\underline{T}_{n+4}(t)=\displaystyle\sum\limits_{i=0}^{2k+3}
\dfrac{(-1)^it^{2i+1}}{(2i+1)!}<
\displaystyle\sum\limits_{i=0}^{\infty}
\dfrac{(-1)^it^{2i+1}}{(2i+1)!}=\sin t
                                                                                             \\[3.5 ex]
\Longleftrightarrow\;
\displaystyle\sum\limits_{j=1}^\infty
\dfrac{-t^{4(k+j)+5}}{{\big(}4(k+j)+5{\big)}!}
\underbrace{\left(1-\dfrac{t^2}{{\big(}4(k+j)+6{\big)}{\big(}4(k+j)+7{\big)}}\right)}_{(>0)}
<0.
\end{array}
$$
The corresponding equalities at the endpoints of the segment $\big{[}0,
\mbox{\small $\sqrt{(n+3)(n+4)}$}\,\big{]}$ are easily checked.
Overall, (\ref{Ineq_4}) has been proved. For $t \in \big{[}\mbox{\small $-\sqrt{(n+3)(n+4)}$},
0\big{)}$, (\ref{Ineq_5}) is valid on the basis of the odd property of the function $\sin x$.
Overall, (\ref{Ineq_5}) has been proved.
\end{proof}
\begin{lemma}
{\rm \textbf{(i)}}
For the polynomial $T_n(t)=\displaystyle\sum\limits_{i=0}^{n/2}\dfrac{(-1)^it^{2i}}{(2i)!}$,
where $n=4k$, $k\in \mathbb{N}_{0}$, it is valid:
\begin{equation}
\label{Ineq_6}
\Big{(}\forall t \in \big{[}\mbox{\small $-\sqrt{(n+3)(n+4)},\sqrt{(n+3)(n+4)}$}\,\big{]}\Big{)} \,\,\,\overline{T}_n(t)\geq \overline{T}_{n+4}(t)\geq \cos t.
\end{equation}

\smallskip
\noindent
For the value $t=0$ the inequalities in $(\ref{Ineq_6})$ turn into equalities.
For the values $t=\mbox{\small $\pm\sqrt{(n+3)(n+4)}$}$ the equality $\overline{T}_n(t)=\overline{T}_{n+4}(t)$ is true.

\smallskip
\noindent
{\rm \textbf{(ii)}}
For the polynomial $T_n(t)=\displaystyle\sum\limits_{i=0}^{n/2}\dfrac{(-1)^it^{2i}}{(2i)!}$,
where $n=4k+2$, $k\in \mathbb{N}_{0}$, it is valid:
\begin{equation}
\label{Ineq_7}
\Big{(}\forall t \in \big{[}\mbox{\small $-\sqrt{(n+3)(n+4)},\sqrt{(n+3)(n+4)}$}\,\big{]}\Big{)} \,\,\, \underline{T}_n(t)\leq \underline{T}_{n+4}(t)\leq \cos t.
\end{equation}

\noindent
For the value $t=0$ the inequalities in $(\ref{Ineq_7})$ turn into equalities.
For the values $t=\mbox{\small $\pm\sqrt{(n+3)(n+4)}$}$ the equality $\underline{T}_n(t)=\underline{T}_{n+4}(t)$ is true.
\end{lemma}
\begin{proof}
{\rm \textbf{(i)}}
Let us suppose that $ 0<t\leq\sqrt{(n+3)(n+4)}$, then it is correct:
$$
\begin{array}{l}
\overline{T}_n(t)=\displaystyle\sum\limits_{i=0}^{2k}\dfrac{(-1)^it^{2i}}{(2i)!}>
\displaystyle\sum\limits_{i=0}^{\infty}\dfrac{(-1)^it^{2i}}{(2i)!}=\cos t
                                                                                            \\[3.5 ex]
\Longleftrightarrow\;
\displaystyle\sum\limits_{j=1}^\infty
\dfrac{t^{4(k+j)-2}}{{\big(}4(k+j)-2{\big)}!}
\underbrace{\left(1-\dfrac{t^2}{{\big(}4(k+j)-1{\big)}{\big(}4(k+j){\big)}}\right)}_{(\geq0)}
>0;
\end{array}
$$
then,
$$
\overline{T}_{n+4}(t)=\overline{T}_n(t)-\dfrac{t^{n+2}}{(n+2)!}
\underbrace{\left(1-\dfrac{t^2}{(n+3)(n+4)}\right)}_{(\geq0)}\leq \overline{T}_n(t)
$$
and
$$
\begin{array}{l}
\overline{T}_{n+4}(t)=\displaystyle\sum\limits_{i=0}^{2k+2}\dfrac{(-1)^it^{2i}}{(2i)!}>
\displaystyle\sum\limits_{i=0}^{\infty}\dfrac{(-1)^it^{2i}}{(2i)!}=\cos t
                                                                                              \\[3.5 ex]
\Longleftrightarrow\;
\displaystyle\sum\limits_{j=1}^\infty
\dfrac{t^{4(k+j)+2}}{{\big(}4(k+j)+2{\big)}!}
\underbrace{\left(1-\dfrac{t^2}{{\big(}4(k+j)+3{\big)}{\big(}4(k+j)+4{\big)}}\right)}_{(>0)}
>0.
\end{array}
$$
The corresponding equalities at the endpoints of the segment $\big{[}0,
\mbox{\small $\sqrt{(n+3)(n+4)}$}\,\big{]}$ are easily checked.
Therefore, the inequalities in (\ref{Ineq_6}) has been proved for
$t \in\big{[}0,\mbox{\small $\sqrt{(n+3)(n+4)}$}\,\big{]}$.
For $t \in \big{[}\mbox{\small $-\sqrt{(n+3)(n+4)}$},0\big{)}$
the inequalities in (\ref{Ineq_6}) are valid on the basis of the
even property of the function $\cos x$.
Overall, (\ref{Ineq_6}) has been proved.

\smallskip
\noindent
{\rm \textbf{(ii)}}
Let us suppose that $0<t\leq\sqrt{(n+3)(n+4)}$, then it is correct:
$$
\begin{array}{l}
\underline{T}_n(t)=\displaystyle\sum\limits_{i=0}^{2k+1}\dfrac{(-1)^it^{2i}}{(2i)!}<
\displaystyle\sum\limits_{i=0}^{\infty}\dfrac{(-1)^it^{2i}}{(2i)!}=\cos t
                                                                                            \\[3.5 ex]
\Longleftrightarrow\;
\displaystyle\sum\limits_{j=1}^\infty
-\dfrac{t^{4(k+j)}}{{\big(}4(k+j){\big)}!}
\underbrace{\left(1-\dfrac{t^2}{{\big(}4(k+j)+1{\big)}{\big(}4(k+j)+2{\big)}}\right)}_{(\geq0)}
<0;
\end{array}
$$
then,
$$
\underline{T}_{n+4}(t)=\underline{T}_n(t)+\dfrac{t^{n+2}}{(n+2)!}
\underbrace{\left(1-\dfrac{t^2}{(n+3)(n+4)}\right)}_{(\geq0)}\geq \underline{T}_n(t)
$$
and
$$
\begin{array}{l}
\underline{T}_{n+4}(t)=\displaystyle\sum\limits_{i=0}^{2k+3}\dfrac{(-1)^it^{2i}}{(2i)!}<
\displaystyle\sum\limits_{i=0}^{\infty}\dfrac{(-1)^it^{2i}}{(2i)!}=\cos t
                                                                                             \\[3.5 ex]
\Longleftrightarrow\;
\displaystyle\sum\limits_{j=1}^\infty
\dfrac{-t^{4(k+j)+4}}{{\big(}4(k+j)+4{\big)}!}
\underbrace{\left(1-\dfrac{t^2}{{\big(}4(k+j)+5{\big)}{\big(}4(k+j)+6{\big)}}\right)}_{(>0)}
<0.
\end{array}
$$
The corresponding equalities at the endpoints of the segment $\big{[}0,
\mbox{\small $\sqrt{(n+3)(n+4)}$}\,\big{]}$ are easily checked.
Therefore, the inequalities in (\ref{Ineq_7}) has been proved for
$t \in\big{[}0,\mbox{\small $\sqrt{(n+3)(n+4)}$}\,\big{]}$.
For $t \in \big{[}\mbox{\small $-\sqrt{(n+3)(n+4)}$},0\big{)}$
the inequalities in (\ref{Ineq_7}) are valid on the basis of the
even property of the function $\cos x$.
Overall, (\ref{Ineq_7}) has been proved.
\end{proof}

\smallskip
\noindent
Let us consider a complex number $z=e^{\eusm{i}x}$ {\big (}$x\!\in\!\mathbb{R},\,\eu{i}\!=\!\sqrt{-1}$ -- imaginary~unit{\big )}.
Then it holds:
\begin{equation}
\label{Ineq_8}
\cos x=\displaystyle\frac{1}{2}{\Big (}z+\frac{1}{z}{\Big )}
\quad\mbox{and}\quad
\sin x=\displaystyle\frac{1}{2\eu{i}}{\Big (}z-\frac{1}{z}{\Big )}.
\end{equation}
Let us introduce the following functions:
\begin{equation}
\label{Ineq_9}
R_{k}(z)=z^{k}+\displaystyle\frac{1}{z^{k}}
\quad\mbox{and}\quad
Q_{k}(z)=z^{k}-\displaystyle\frac{1}{z^{k}},
\end{equation}
for $k\!=\!1,2,\ldots\;$. Then it is:
\begin{equation}
\label{Ineq_10}
R_{k}(z)=2\cos (kx)
\quad\mbox{and}\quad
Q_{k}(z)=2\eu{i}\sin (kx),
\end{equation}
for $z\!=\!e^{\eusm{i}x}\!$ and $k\!=\!1,2,\ldots\;$. Hence, we may come to the conclusion that it holds:
\begin{equation}
\label{Ineq_11}
R_{n}(z)\cdot R_{m}(z)=R_{n+m}(z)+R_{\mid n-m \mid}(z)
\end{equation}
and
\begin{equation}
\label{Ineq_12}
R_{n}(z)\cdot Q_{m}(z)=Q_{n+m}(z)+\nu \cdot Q_{\mid n-m \mid}(z),
\end{equation}
where $\nu=\mbox{\rm sgn}{\big (}m-n{\big )}$. Specifically, $R_{0}(z)=2$ and $Q_{0}(z)=0$.

\smallskip
\noindent
In the following auxiliary proposition we show that $\sin^{n}\!x$ can be presented as a sum of sines of multiple angels
or sum of cosines of multiple angels depending on the parity of degree $n$.
\begin{lemma}
For $n \in \mathbb{N}$ the following formulas are valid:

\smallskip
\noindent
{\rm \textbf{(i)}}
if $n$ is odd
\begin{equation}
\label{Ineq_13}
\sin^{n}\!x=\frac{2}{2^{n}}\displaystyle\sum\limits_{k=0}^{\frac{n-1}{2}}(-1)^{\frac{n-1}{2}+k}
\NBC{\!n\!}{\!k\!} \sin{\big (}(n-2k)x{\big )},
\end{equation}

\smallskip
\noindent
{\rm \textbf{(ii)}}
if $n$ is even
\begin{equation}
\label{Ineq_14}
\sin^{n}\!x=\frac{1}{2^{n}} \nbc{\!\!n\!\!}{\!\!\frac{n}{2}\!\!}+
\frac{2}{2^{n}} \displaystyle\sum\limits_{k=0}^{\frac{n}{2}-1}(-1)^{\frac{n}{2}+k}
\NBC{\!n\!}{\!k\!} \cos {\big (}(n-2k)x{\big )}.
\end{equation}
\end{lemma}
\begin{proof}
See Ex. 17, 18, Chapter IX \cite{Durell_Robson_1930} and the method of proving from \cite{Krechmar_1978}.
\end{proof}

\noindent
For the function $\cos^{n}\!x$ the following proposition analogously holds:
\begin{lemma}
For $n \in \mathbb{N}$ the following formulas are valid:

\smallskip
\noindent
{\rm \textbf{(i)}}
if $n$ is odd
\begin{equation}
\label{Ineq_15}
\cos^{n}\!x=\frac{2}{2^{n}}\displaystyle\sum\limits_{k=0}^{\frac{n-1}{2}} \NBC{\!n\!}{\!k\!}
\cos{\big(}(n-2k)x{\big)},
\end{equation}

\smallskip
\noindent
{\rm \textbf{(ii)}}
if $n$ is even
\begin{equation}
\label{Ineq_16}
\cos^{n}\!x=\frac{1}{2^{n}} \nbc{\!\!n\!\!}{\!\!\frac{n}{2}\!\!} +\frac{2}{2^{n}}
\displaystyle\sum\limits_{k=0}^{\frac{n}{2}-1} \NBC{\!n\!}{\!k\!}
\cos {\big(}(n-2k)x{\big)}.
\end{equation}
\end{lemma}
\begin{proof}
See Ex. 15, 16, Chapter IX \cite{Durell_Robson_1930} and the method of proving from \cite{Krechmar_1978}.
\end{proof}

\smallskip
\noindent
Based on the previous two Lemmas we give a proof of the following statement:

\begin{theorem}
For $n, m \in \mathbb{N}$ we have the following cases:

\smallskip
\noindent
{\rm \textbf{(i)}} if both $n$ and $m$ are odd
\begin{equation}
\label{Ineq_17}
\begin{array}{rcl}
\cos^{n}\!x\! \cdot \sin^{m}\!x\!
\!\!&\!\!=\!\!&\!\!
\frac{1}{2^{n+m-1}}\!\!\!\!\!\displaystyle\sum\limits_{k=0}^{\frac{n+m}{2}-\mbox{\tiny $1$}}
\!\!\!(\!-1\!)^{\frac{m-1}{2}+k}\!
{\Bigg (}
\!\displaystyle\sum\limits_{r=0}^{k}(\!-1\!)^{r} \NBC{\!n\!}{\!r\!} \NBC{\!m\!}{\!k-r\!}
\sin \!{\big (}\!(\!n\!+\!m\!-\!2k\!)x\!{\big )}\!
{\Bigg )},
\end{array}
\end{equation}

\smallskip
\noindent
{\rm \textbf{(ii)}} if $n$ is even and $m$ is odd
\begin{equation}
\label{Ineq_18}
\begin{array}{rcl}
\cos^{n}\!x\! \cdot \sin^{m}\!x\!
\!\!&\!\!=\!\!&\!\!
\frac{1}{2^{n+m-1}}\!\!\!\!\!\displaystyle\sum\limits_{k=0}^{\frac{n+m-1}{2}}
\!\!\!(\!-1\!)^{\frac{m-1}{2}+k}\!
{\Bigg (}
\!\displaystyle\sum\limits_{r=0}^{k}(\!-1\!)^{r} \NBC{\!n\!}{\!r\!} \NBC{\!m\!}{\!k-r\!}
\sin \!{\big (}\!(\!n\!+\!m\!-\!2k\!)x\!{\big )}\!
{\Bigg )},
\end{array}
\end{equation}

\smallskip
\noindent
{\rm \textbf{(iii)}} if $n$ is odd and $m$ is even
\begin{equation}
\label{Ineq_19}
\begin{array}{rcl}
\cos^{n}\!x\! \cdot \sin^{m}\!x\!
\!\!&\!\!=\!\!&\!\!
\frac{1}{2^{n+m-1}}\!\!\!\!\!\displaystyle\sum\limits_{k=0}^{\frac{n+m-1}{2}}
\!\!\!(\!-1\!)^{\frac{m}{2}+k}\!
{\Bigg (}
\!\displaystyle\sum\limits_{r=0}^{k}(\!-1\!)^{r} \NBC{\!n\!}{\!r\!} \NBC{\!m\!}{\!k-r\!}
\cos \!{\big (}\!(\!n\!+\!m\!-\!2k\!)x\!{\big )}\!
{\Bigg )},
\end{array}
\end{equation}

\smallskip
\noindent
{\rm \textbf{(iv)}} if both $n$ and $m$ are even
\begin{equation}
\label{Ineq_20}
\begin{array}{rcl}
\cos^{n}\!x\! \cdot \sin^{m}\!x\!
\!\!&\!\!=\!\!&\!\!
\frac{1}{2^{n+m-1}}
{\Bigg(}\!\!
\displaystyle\sum\limits_{k=0}^{\frac{n+m}{2}-\mbox{\tiny $1$}}\!\!\!(\!-1\!)^{\frac{m}{2}+k}
\!\displaystyle\sum\limits_{r=0}^{k}(\!-1\!)^{r} \NBC{\!n\!}{\!r\!} \NBC{\!m\!}{\!k-r\!}
\cos \!{\big (}\!(\!n\!+\!m\!-\!2k\!)x\!{\big )}\!
                                                                                  \\[1.5ex]
\!\!&\!\!+\!\!&\!\!
\frac{1}{2}(-1)^{\frac{2m+n}{2}}\displaystyle\sum\limits_{r=0}^{\frac{n+m}{2}}(-1)^{r}
\NBC{\!n\!}{\!r\!}\nbc{\!m\!}{\!\!\frac{n+m}{2}-r\!\!}\!\!
{\Bigg)}.
\end{array}
\end{equation}
\end{theorem}

\vspace{-0.2 cm}
\noindent
\begin{proof}
{\rm \textbf{(i)}} Let us suppose that $n$ and $m$ are both odd, then:
$$
\begin{array}{l}
\cos^{n}x\cdot\sin^{m}x=
                                                                                   \\[1.0 ex]
= {\Bigg (}\!\!
\frac{2}{2^{n}}\displaystyle\sum\limits_{i=0}^{\frac{n-1}{2}}
\NBC{\!n\!}{\!i\!} \cos{\big (}(n-2i)x{\big )}
\!\!{\Bigg )}
\cdot
{\Bigg (}\!\!
\mbox{\small $\frac{2}{2^{m}}$}\displaystyle\sum\limits_{j=0}^{\frac{m-1}{2}}(-1)^{\frac{m-1}{2}+j} \NBC{\!m\!}{\!j\!} \sin{\big (}(m-2j)x{\big )} \!\!
{\Bigg )}                                                                          \\[2.0 ex]
=\frac{1}{2^{n+m}\eusm{i}}
{\Bigg (}
\displaystyle\sum\limits_{i=0}^{\frac{n-1}{2}} \NBC{\!n\!}{\!i\!}R_{n-2i}(z)
{\Bigg )}
\cdot
{\Bigg (}
\displaystyle\sum\limits_{j=0}^{\frac{m-1}{2}}(-1)^{\frac{m-1}{2}+j}\NBC{\!m\!}{\!j\!}Q_{m-2j}(z)
{\Bigg )}
                                                                                   \\[2.0 ex]
= \frac{1}{2^{n+m}\eusm{i}}
\displaystyle\sum\limits_{i=0}^{\frac{n-1}{2}}
\displaystyle\sum\limits_{j=0}^{\frac{m-1}{2}} (-1)^{\frac{m-1}{2}+j}\NBC{\!n\!}{\!i\!}
\NBC{\!m\!}{\!j\!}R_{n-2i}(z)\cdot Q_{m-2j}(z)
                                                                                    \\
=\frac{1}{2^{n+m}\eusm{i}}
\displaystyle\sum\limits_{i=0}^{\frac{n-1}{2}}
\displaystyle\sum\limits_{j=0}^{\frac{m-1}{2}} (-1)^{\frac{m-1}{2}+j}\NBC{\!n\!}{\!i\!}\NBC{\!m\!}{\!j\!}
{\Big (} \!
Q_{n+m-2(i+j)}(z)+\nu Q_{\mid n-m-2(i-j)\mid }  \!
{\Big )}
\end{array}
$$

\break

$$
\begin{array}{l}
= \frac{1}{2^{n+m}\eusm{i}}
{\Bigg (}
\displaystyle\sum\limits_{i=0}^{\frac{n-1}{2}}
\displaystyle\sum\limits_{j=0}^{\frac{m-1}{2}} (-1)^{\frac{m-1}{2}+j} \NBC{\!n\!}{\!i\!}
\NBC{\!m\!}{\!j\!}Q_{n+m-2(i+j)}(z)
                                                                                   \\[3.5ex]
\;\;\;\;\;\;\;\;\;\;\;\;\;+
\displaystyle\sum\limits_{i=0}^{\frac{n-1}{2}}\sum\limits_{j=0}^{\frac{m-1}{2}}
(-1)^{\frac{m-1}{2}+j}\NBC{\!n\!}{\!i\!}\NBC{\!m\!}{\!j\!}\nu Q_{\mid n-m-2(i-j)\mid }(z)
{\Bigg )},
                                                                                    \\[3.5ex]
\end{array}
$$
where $\nu\!=\!\mbox{\rm sgn}{\big (}m\!-\!n\!-\!2(j\!-\!i){\big )}$. Now we observe
the binomial coefficients next to $Q_{\mid n-m-2(i-j)\mid }(z)$. The products
of binomial coefficients may be written in the following way:
\begin{equation}
\label{Ineq_21}
\NBC{\!n\!}{\!i\!}\NBC{\!m\!}{\!j\!}=\NBC{\!n\!}{\!n-i\!}\NBC{\!m\!}{\!j\!}=\NBC{\!n\!}{\!i\!}\NBC{\!m\!}{\!m-j\!}=\NBC{\!n\!}{\!n-i\!}\NBC{\!m\!}{\!m-j\!}.
\end{equation}
We may notice that the sums of the lower numbers in the products of the binomial coefficients
of the previous equalities in (\ref{Ineq_21}) are $i+j$, $n-i+j$, $i+m-j$ and $n-i+m-j$. Let us mark the index
$\vert n-m-2(i-j) \vert $ with $d$. Our aim is to determine $k$ in such a way that $n+m-2k=d$.
Then we see two possibilities:

\medskip

\smallskip
1) when $n-m-2(i-j)>0$, then
\begin{equation}
\label{Ineq_22}
n+m-2k=n-m-2(i-j)\Longrightarrow k=m+i-j,
\end{equation}

\smallskip
2) when $n-m-2(i-j)<0$, then
\begin{equation}
\label{Ineq_23}
n+m-2k=-n+m-2(j-i)\Longrightarrow k=n-i+j.
\end{equation}

\vspace*{-1.0 mm}

\noindent
Therefore, while calculating
$
\displaystyle\sum\limits_{i=0}^{\frac{n-1}{2}}\sum\limits_{j=0}^{\frac{m-1}{2}}
(-1)^{\frac{m-1}{2}+j}\NBC{\!n\!}{\!i\!}\NBC{\!m\!}{\!j\!}\nu Q_{\mid n-m-2(i-j)\mid }(z)
$, on the basis of the implication (\ref{Ineq_22}), we~chose the product of the binomial coefficients
$\NBC{\!n\!}{\!i\!}\!\NBC{\!\!m\!\!}{\!m\!-\!j\!}$, and on the basis of the implication
(\ref{Ineq_23}), we chose $\NBC{\!n\!}{\!n-i\!}\!\NBC{\!m\!}{\!j\!}$,
i.e. we chose that product of binomial coefficients whose sum of the lower numbers equals to $k$.

\smallskip
\noindent
Finally, we get the requested result:
$$
\!\begin{array}{rcl}
\cos^nx \cdot \sin^mx
\!\!&\!\!=\!\!&\!\!
\frac{1}{2^{n+m}\eusm{i}}\!\!\displaystyle\sum\limits_{k=0}^{\frac{n+m}{2}-1}
(-1)^{\frac{m-1}{2}+k}
{\Bigg (}  \!
\displaystyle\sum\limits_{r=0}^{k}(-1)^{r} \NBC{\!n\!}{\!r\!} \! \NBC{\!\!m\!\!}{\!\!k-r\!\!}
Q_{n+m-2k}(z)  \!\!
{\Bigg )}
                                                                                    \\[3.5ex]
\!\!&\!\!=\!\!&\!\!
\frac{1}{2^{n+m-1}}\!\!\!\!\!\displaystyle\sum\limits_{k=0}^{\frac{n+m}{2}-\mbox{\tiny $1$}}
\!\!\!(\!-1\!)^{\frac{m-1}{2}+k}\!
{\Bigg (} \!
\displaystyle\sum\limits_{r=0}^{k}(\!-1\!)^{r} \!\NBC{\!n\!}{\!r\!} \! \NBC{\!\!m\!\!}{\!\!k-r\!\!}
\!\sin \!{\big (}(n\!+\!m\!-\!2k)x{\big )} \!\!
{\Bigg )} \!.
\end{array}
$$

\smallskip
\noindent
{\rm \textbf{(ii)}} Let $n$ be even and $m$ odd, then:
$$
\begin{array}{l}
\cos^{n}x\cdot\sin^{m}x=
                                                                                   \\[1.0 ex]
=
{\Bigg (}   \!\!
\frac{1}{2^{n}} \nbc{\!n\!}{\!\frac{n}{2}\!}+\frac{2}{2^{n}}
\displaystyle\sum\limits_{i=0}^{\frac{n}{2}-1} \NBC{\!n\!}{\!i\!} \cos{\big (}(n-2i)x{\big )}  \!\!
{\Bigg )}
                                                                                   \\[2.0 ex]
\;\;\;\cdot
{\Bigg (}  \!\!
\frac{2}{2^{m}}\displaystyle\sum\limits_{j=0}^{\frac{m-1}{2}}(-1)^{\frac{m-1}{2}+j} \NBC{\!m\!}{\!j\!} \sin{\big (}(m-2j)x{\big )} \!\!
{\Bigg )}                                                                          \\[2.0 ex]
=\frac{1}{2^{n+m}\eusm{i}}
{\Bigg (}   \!\!
\nbc{\!n\!}{\!\frac{n}{2}\!}
\displaystyle\sum\limits_{j=0}^{\frac{m-1}{2}}(-1)^{\frac{m-1}{2}+j}\NBC{\!m\!}{\!j\!}Q_{m-2j}(z)
                                                                                   \\[2.0 ex]
\;\;\;\;\;\;\;\;\;\;\;\;\;+
{\Big (}
\displaystyle\sum\limits_{i=0}^{\frac{n}{2}-1} \NBC{\!n\!}{\!i\!}R_{n-2i}(z)
{\Big )}
\cdot
{\Big (}
\displaystyle\sum\limits_{j=0}^{\frac{m-1}{2}}(-1)^{\frac{m-1}{2}+j}\NBC{\!m\!}{\!j\!}Q_{m-2j}(z)
{\Big )}   \!\!
{\Bigg )}
                                                                                      \\[2.0 ex]
=\frac{1}{2^{n+m}\eusm{i}}
{\Bigg (}   \!\!
\nbc{\!n\!}{\!\frac{n}{2}\!}
\displaystyle\sum\limits_{j=0}^{\frac{m-1}{2}}(-1)^{\frac{m-1}{2}+j}\NBC{\!m\!}{\!j\!}Q_{m-2j}(z)
                                                                                   \\[2.0 ex]
\;\;\;\;\;\;\;\;\;\;\;\;\;+
\displaystyle\sum\limits_{i=0}^{\frac{n}{2}-1}\sum\limits_{j=0}^{\frac{m-1}{2}} (-1)^{\frac{m-1}{2}+j}\NBC{\!n\!}{\!i\!}\NBC{\!m\!}{\!j\!}R_{n-2i}(z)
\cdot
Q_{m-2j}(z)  \!\!
{\Bigg )}
                                                                                     \\[2.0 ex]
= \frac{1}{2^{n+m}\eusm{i}}
{\Bigg (}  \!\!
\nbc{\!n\!}{\!\frac{n}{2}\!}
\displaystyle\sum\limits_{j=0}^{\frac{m-1}{2}}(-1)^{\frac{m-1}{2}+j}\NBC{\!m\!}{\!j\!}Q_{m-2j}(z)
                                                                                    \\[2.0 ex]
\;\;\;\;\;\;\;\;\;\;\;\;\;+
\displaystyle\sum\limits_{i=0}^{\frac{n}{2}-1}\sum\limits_{j=0}^{\frac{m-1}{2}} (-1)^{\frac{m-1}{2}+j}\NBC{\!n\!}{\!i\!}\NBC{\!m\!}{\!j\!}
{\Big (}
Q_{n+m-2(i+j)}(z)
+
\nu Q_{\mid n-m-2(i-j)\mid}(z)
{\Big )}  \!\!
{\Bigg )}
                                                                                       \\[2.0 ex]
=
\frac{1}{2^{n+m}\eusm{i}}
{\Bigg (} \!\!
\nbc{\!n\!}{\!\frac{n}{2}\!}
\displaystyle\sum\limits_{j=0}^{\frac{m-1}{2}}(-1)^{\frac{m-1}{2}+j}\NBC{\!m\!}{\!j\!}Q_{m-2j}(z)
                                                                                       \\[2.0 ex]
\;\;\;\;\;\;\;\;\;\;\;\;\;+
\displaystyle\sum\limits_{i=0}^{\frac{n}{2}-1}\sum\limits_{j=0}^{\frac{m-1}{2}} (-1)^{\frac{m-1}{2}+j}\NBC{\!n\!}{\!i\!}\NBC{\!m\!}{\!j\!}Q_{n+m-2(i+j)}(z)
                                                                                   \\[2.0 ex]
\;\;\;\;\;\;\;\;\;\;\;\;\;+
\displaystyle\sum\limits_{i=0}^{\frac{n}{2}-1}\sum\limits_{j=0}^{\frac{m-1}{2}}(-1)^{\frac{m-1}{2}+j}\NBC{\!n\!}{\!i\!}\NBC{\!m\!}{\!j\!}\nu Q_{\mid n-m-2(i-j)\mid }(z)  \!\!
{\Bigg )},
\end{array}
$$
where $\nu=\mbox{\rm sgn}{\big (}m-n-2(j-i){\big )}$. Looking at the products of the binomial coefficients
next to $Q_{\mid n-m-2(i-j)\mid }(z)$, analogously to the equalities (\ref{Ineq_21}) and the procedure
with the implications (\ref{Ineq_22}) and (\ref{Ineq_23}), we may conclude that it is valid:

\smallskip
$$
\begin{array}{rcl}
\cos^nx \cdot \sin^mx
        \!\!&\!\!=\!\!&\!\!
\frac{1}{2^{n+m}\eusm{i}}\!\!\!\!\!\displaystyle\sum\limits_{k=0}^{\frac{n+m-1}{2}}
\!(\!-1\!)^{\frac{m-1}{2}+k}
{\Bigg (}      \!
\displaystyle\sum\limits_{r=0}^{k}\!(-1)^{r} \!\NBC{\!\!n\!\!}{\!\!r\!\!}\!\NBC{\!\!m\!\!}{\!\!k-r\!\!}
Q_{n+m-2k}(z)  \!\!
{\Bigg )}                                                                          \\[3.0 ex]
          \!\!&\!\!=\!\!&\!\!
\frac{1}{2^{n+m-1}}\!\!\!\!\displaystyle\sum\limits_{k=0}^{\frac{n+m-1}{2}}
\!(\!-1\!)^{\frac{m-1}{2}+k} \!
{\Bigg (}      \!
\displaystyle\sum\limits_{r=0}^{k}\!(-1)^{r}\!\NBC{\!\!n\!\!}{\!\!r\!\!}\!\NBC{\!\!m\!\!}{\!\!k-r\!\!}
\!\sin\!{\big (}(\!n\!+\!m\!-\!2k\!)x{\big )}  \!\!
{\Bigg )}\!,
\end{array}
$$

\smallskip
\noindent
{\rm \textbf{(iii)}} Replacing $x$ by $\frac{\pi}{2}-x$ in formula (\ref{Ineq_18}), we get formula (\ref{Ineq_19}).

\medskip
\noindent
{\rm \textbf{(iv)}} If $n$ and $m$ are both even, then:
$$
\begin{array}{l}
\cos^{n}x\cdot\sin^{m}x=
                                                                                       \\[1.0 ex]
={\Bigg (}  \!\!
\frac{1}{2^{n}}\nbc{\!n\!}{\!\frac{n}{2}\!}
+\frac{2}{2^{n}}\displaystyle\sum\limits_{i=0}^{\frac{n}{2}-1} \NBC{\!n\!}{\!i\!}
\cos{\big (}(n-2i)x{\big )}   \!\!
{\Bigg )}                                                                              \\[3.0 ex]
\;\;\cdot
{\Bigg (} \!\!
\frac{1}{2^{m}}\nbc{\!m\!}{\!\frac{m}{2}\!}
+\frac{2}{2^{m}}\displaystyle\sum\limits_{j=0}^{\frac{m}{2}-1}(-1)^{\frac{m}{2}+j} \NBC{\!m\!}{\!j\!} \cos{\big(}(m-2j)x{\big)} \!\!
{\Bigg )}
                                                                                          \\[3.0 ex]
=\frac{1}{2^{n+m}}
{\Bigg (} \!\!
\nbc{\!n\!}{\!\!\frac{n}{2}\!\!} \! \nbc{\!m\!}{\!\!\frac{m}{2}\!\!}\!+\!\nbc{\!n\!}{\!\frac{n}{2}\!} \!\!\displaystyle\sum\limits_{j=0}^{\frac{m}{2}-1}(-1)^{\frac{m}{2}+j} \NBC{\!m\!}{\!j\!}R_{m-2j}(z)
                                                                                       \\[3.0 ex]
\;\;\;\;\;\;\;\;\;\;\;\;+\,\nbc{\!m\!}{\!\!\frac{m}{2}\!\!}\!\!\displaystyle\sum\limits_{i=0}^{\frac{n}{2}-1} \NBC{\!n\!}{\!i\!}R_{n-2i}(z)
                                                                                       \\[3.0 ex]
\;\;\;\;\;\;\;\;\;\;\;\;+
{\Big (}
\displaystyle\sum\limits_{i=0}^{\frac{n}{2}-1}\NBC{\!n\!}{\!i\!}R_{n-2i}(z)
{\Big )}
\!\cdot\!
{\Big (}
\displaystyle\sum\limits_{j=0}^{\frac{m}{2}-1}(-1)^{\frac{m}{2}+j}\NBC{\!m\!}{\!j\!}R_{m-2j}(z)
{\Big )}  \!\!
{\Bigg )}
                                                                                        \\[3.0 ex]
=\frac{1}{2^{n+m}}
{\Bigg (}\!\!
\nbc{\!n\!}{\!\!\frac{n}{2}\!\!} \nbc{\!m\!}{\!\!\frac{m}{2}\!\!}\!+\!\nbc{\!n\!}{\!\!\frac{n}{2}\!\!} \!\!\displaystyle\sum\limits_{j=0}^{\frac{m}{2}-1}(-1)^{\frac{m}{2}+j} \NBC{\!m\!}{\!j\!}R_{m-2j}(z)
\!+\!\nbc{\!m\!}{\!\!\frac{m}{2}\!\!}\!\!\displaystyle\sum\limits_{i=0}^{\frac{n}{2}-1}
\NBC{\!n\!}{\!i\!}R_{n-2i}(z)
                                                                                       \\[3.0 ex]
\;\;\;\;\;\;\;\;\;\;\;\;+
\displaystyle\sum\limits_{i=0}^{\frac{n}{2}-1}
\displaystyle\sum\limits_{j=0}^{\frac{m}{2}-1}(-1)^{\frac{m}{2}+j}\NBC{\!n\!}{\!i\!}
\NBC{\!m\!}{\!j\!}R_{n-2i}(z)
\cdot
R_{m-2j}(z)  \!\!
{\Bigg )}
                                                                                         \\[3.0 ex]
=\frac{1}{2^{n+m}}
{\Bigg (}  \!\!
\nbc{\!n\!}{\!\!\frac{n}{2}\!\!} \nbc{\!m\!}{\!\!\frac{m}{2}\!\!}\!+\!\nbc{\!n\!}{\!\!\frac{n}{2}\!\!} \!\!\displaystyle\sum\limits_{j=0}^{\frac{m}{2}-1}(-1)^{\frac{m}{2}+j} \NBC{\!m\!}{\!j\!}R_{m-2j}(z)
\!+\!\nbc{\!m\!}{\!\!\frac{m}{2}\!\!}
\!\!\displaystyle\sum\limits_{i=0}^{\frac{n}{2}-1} \NBC{\!n\!}{\!i\!}R_{n-2i}(z)
                                                                                       \\[3.0 ex]
\;\;\;\;\;\;\;\;\;\;\;\;+
\displaystyle\sum\limits_{i=0}^{\frac{n}{2}-1}
\displaystyle\sum\limits_{j=0}^{\frac{m}{2}-1}(-1)^{\frac{m}{2}+j}\NBC{\!n\!}{\!i\!}\NBC{\!m\!}{\!j\!}
{\Big (}
R_{n+m-2(i+j)}(z)+R_{\mid n-m-2(i-j) \mid}(z)
{\Big )}  \!\!
{\Bigg )}
                                                                                       \\[3.0 ex]
=\frac{1}{2^{n+m}}
{\Bigg (} \!\!
\nbc{\!n\!}{\!\!\frac{n}{2}\!\!} \nbc{\!m\!}{\!\!\frac{m}{2}\!\!}\!+\!\nbc{\!n\!}{\!\!\frac{n}{2}\!\!}\!\!\displaystyle\sum\limits_{j=0}^{\frac{m}{2}-1}(-1)^{\frac{m}{2}+j}\NBC{\!m\!}{\!j\!}R_{m-2j}(z)
\!+\!\nbc{\!m\!}{\!\!\frac{m}{2}\!\!}\!\!\displaystyle\sum\limits_{i=0}^{\frac{n}{2}-1}\NBC{\!n\!}{\!i\!} R_{n-2i}(z)
                                                                                       \\[3.0 ex]
\;\;\;\;\;\;\;\;\;\;\;\;+
\displaystyle\sum\limits_{i=0}^{\frac{n}{2}-1}
\displaystyle\sum\limits_{j=0}^{\frac{m}{2}-1}(-1)^{\frac{m}{2}+j}\NBC{\!n\!}{\!i\!}
\NBC{\!m\!}{\!j\!}R_{n+m-2(i+j)}(z)
                                                                                         \\[3.0 ex]
\;\;\;\;\;\;\;\;\;\;\;\;+
\displaystyle\sum\limits_{i=0}^{\frac{n}{2}-1}
\displaystyle\sum\limits_{j=0}^{\frac{m}{2}-1}(-1)^{\frac{m}{2}+j}\NBC{\!n\!}{\!i\!}
\NBC{\!m\!}{\!j\!}R_{\mid n-m-2(i-j) \mid}(z) \!\!
{\Bigg )}.
                                                                                       \\[3.0 ex]
\end{array}
$$
Looking at the products of the binomial coefficients next to $R_{\mid n-m-2(i-j)\mid }(z)$,
analogously to the equalities (\ref{Ineq_21}) and the procedure with the implications
(\ref{Ineq_22}) and (\ref{Ineq_23}), we may conclude that it is valid:

\newpage

$$
\begin{array}{rcl}
\cos^nx \cdot \sin^mx
                 \!\!&\!\!=\!\!&\!\!
\frac{1}{2^{n+m}}
{\Bigg (}  \!
\displaystyle\sum\limits_{k=0}^{\frac{n+m}{2}-1}(-1)^{\frac{m}{2}+k}
\displaystyle\sum\limits_{r=0}^{k}(-1)^{r}\NBC{\!n\!}{\!r\!}\NBC{\!m\!}{\!k-r\!}R_{n+m-2k}(z)
                                                                                      \\[2.5 ex]
\!\!&\!\!+\!\!&\!\!
\frac{1}{2}(-1)^{\frac{2m+n}{2}}
\displaystyle\sum\limits_{r=0}^{\frac{n+m}{2}}(-1)^{r}\NBC{\!n\!}{\!r\!}
\nbc{\!m\!}{\!\!\frac{n+m}{2}-r\!\!}R_{0} \!\!
{\Bigg )}
                                                                                       \\[2.5ex]
                  \!\!&\!\!=\!\!&\!\!
\frac{1}{2^{n+m-1}} \!
{\Bigg (}  \!
\displaystyle\sum\limits_{k=0}^{\frac{n+m}{2}-1}\!\!\!\!(\!-1\!)^{\frac{m}{2}+k}
\displaystyle\sum\limits_{r=0}^{k}(\!-1\!)^{r}
\!\NBC{\!\!n\!\!}{\!\!r\!\!}\!\NBC{\!\!m\!\!}{\!\!k-r\!\!}
\!\cos\!{\big (}(n\!+\!m\!-\!2k)x{\big )}
                                                                                       \\[2.5ex]
\!\!&\!\!+\!\!&\!\!
\frac{1}{2}(-1)^{\frac{2m+n}{2}}
\displaystyle\sum\limits_{r=0}^{\frac{n+m}{2}}(-1)^{r}\NBC{\!n\!}{\!r\!}
\nbc{\!\!m\!\!}{\!\!\frac{n+m}{2}-r\!\!}  \!\!
{\Bigg )},
                                                                                       \\[2.5ex]
\end{array}
$$ with the note that $\NBC{\!n\!}{\!i\!}\!\NBC{\!m\!}{\!j\!}\cdot R_{0}$ ($R_{0}=2$)
is written as a sum of two products of binomial coefficients equal to
$\NBC{\!n\!}{\!i\!}\!\NBC{\!m\!}{\!j\!}$, analogously to (\ref{Ineq_21}), whose sum of the lower numbers
equals to $\frac{n+m}{2}$.
\end{proof}

\section{A description of the method} 

\smallskip
\noindent
\textbf{I}\enskip\enskip
Our aim is to form a method of proving inequalities of the type (\ref{Ineq_1}) for \mbox{$x \in (0,\delta)$}
and $\delta = \delta_2 \!> \!0$. We will be using upward and downward Maclaurin approximations
of the sine and cosine functions determined in the Lemmas 1.1. and 1.2.

\smallskip
\noindent
\qquad
Let us observe the addend of the sum (\ref{Ineq_1}):
$
s_{i}(x)=
\alpha_{i}x^{p_{i}}\!\cos^{q_{i}}\!\!x\sin^{r_{i}}\!\!x,
$
where $\alpha_{i} \neq 0$ for $i=1,\ldots,n$.
Let us introduce the symbol
\begin{equation}
\label{Ineq_34}
\eu{m}_{i}\!=\!\left\{\!\!\!
\begin{array}{ll}
\frac{q_{i}+r_{i}}{2}\!-\!1, \!\!\!& \!\!\! \mbox{ when } \!q_{i}\! \mbox{ and } \!r_{i}\!
\mbox{ are\! both\! even\! or\! both\! odd,}
                                                                                            \\ [1.5ex]
\frac{q_{i}+r_{i}-1}{2}, \!\!\!\!\!\!& \!\!\!\!\!\! \mbox{ when } \!q_{i}\! \mbox{ and } \!r_{i}\!
\mbox{ have\! different\! parity.}
\end{array}
\right.
\end{equation}

\smallskip
\noindent
According to the Theorem 1.5. the addends $s_{i}(x)$ $(i\!=\!1,2,\ldots,n)$ are represented in four different ways
depending on the cases, so the following possibilities are given in the description of the method:

\medskip
\noindent
\textbf{1.} Let $q_{i}$ and $r_{i}$ be odd or let $q_{i}$ be even and $r_{i}$ odd. In both cases, it holds:
\begin{equation}
\label{Ineq_35}
\hspace*{-0.1 cm}
\begin{array}{rcl}
s_{i}(x)
           & \!\!\!\! = \!\!\!\!&
\alpha_{i}x^{p_{i}}\!\cos^{q_{i}}\!\!x\sin^{r_{i}}\!\!x \\ [1.5ex]
           & \!\!\!\! = \!\!\!\!&
\frac{\alpha_{i}x^{p_{i}}}{2^{q_{i}+r_{i}-1}} \!\!
\displaystyle\sum\limits_{k=0}^{\euft{m}_{i}}(-1)^{\frac{r_{i}-1}{2}+k} \!
\displaystyle\sum\limits_{r=0}^{k}(-1)^{r}\NBC{\!\!q_{i}\!\!}{\!\!r\!\!}\NBC{\!\!r_{i}\!\!}{\!\!k-r\!\!} \!\sin\!{\big(}\!(\!q_{i}\!+\!r_{i}\!-\!2k\!)x\!{\big)}
                                                                                       \\[1.5ex]
           & \!\!\!\! = \!\!\!\!&
\frac{x^{p_{i}}}{2^{q_{i}+r_{i}-1}} \!\!
\displaystyle\sum\limits_{k=0}^{\euft{m}_{i}}
\!\!{\Bigg(} \!
\displaystyle\sum\limits_{r=0}^{k}\alpha_{i}(-1)^{\frac{r_{i}-1}{2}+k+r}
\NBC{\!\!q_{i}\!\!}{\!\!r\!\!}\!\NBC{\!\!r_{i}\!\!}{\!\!k-r\!\!} \!\!
{\Bigg)}
\!\sin\!{\big(}\!(\!q_{i}\!+\!r_{i}\!-\!2k\!)x\!{\big)}.
\end{array}
\end{equation}
Let us mark with
$\beta_{k}=\displaystyle\sum\limits_{r=0}^{k}\alpha_{i}(-1)^{\frac{r_{i}-1}{2}+k+r}
\NBC{\!\!q_{i}\!\!}{\!\!r\!\!}\NBC{\!\!r_{i}\!\!}{\!\!k-r\!\!}$. Then, for every sub-addend $\beta_{k}\sin\!{\big(}(q_{i}+r_{i}-2k)x{\big)}$, depending on the sign of $\beta_{k}$, two cases are possible:

\medskip
\noindent
\quad
1) if $\beta_{k}>0$:
\begin{equation}
\label{Ineq_36}
\beta_{k}\sin\!{\big(}(q_{i}+r_{i}-2k)x{\big)}>
\beta_{k}\underline{T}^{\sin,0}_{\,4l^{(i)}_{k}+3}{\big(}(q_{i}+r_{i}-2k)x{\big)},
\end{equation}

\smallskip
\noindent
\quad
2) if $\beta_{k}<0$:
\begin{equation}
\label{Ineq_37}
\beta_{k}\sin\!{\big(}(q_{i}+r_{i}-2k)x{\big)}>
\beta_{k}\overline{T}^{\,\sin,0}_{4l^{(i)}_{k}+1}{\big(}(q_{i}+r_{i}-2k)x{\big)}.
\end{equation}

\smallskip
\noindent
Let the addend $s_{i}(x)$ be written in the form:
\begin{equation}
\label{Ineq_38}
s_{i}(x)=\frac{x^{p_{i}}}{2^{q_{i}+r_{i}-1}}
\displaystyle\sum\limits_{k=0}^{\euft{m}_{i}}\beta_{k}\sin\!{\big(}(q_{i}+r_{i}-2k)x{\big)}.
\end{equation}
Then it holds:
\vspace{-0.3 cm}
\begin{equation}
\label{Ineq_39}
s_{i}(x)
>
\mbox{\large $\tau$}_{i}(x) = \frac{x^{p_{i}}}{2^{q_{i}+r_{i}-1}}
\displaystyle\sum\limits_{k=0}^{\euft{m}_{i}}
\beta_{k}T^{\sin,0}_{4l^{(i)}_{k}+u}{\big(}(q_{i}+r_{i}-2k)x{\big)},
\end{equation}
where
$
u=\left\{
\begin{array}{ll}
3, & \beta_{k}>0, \\ [1.5ex]
1, & \beta_{k}<0
\end{array}
\right.
$, $l^{(i)}_{k} \in \mathbb{N}_{0}$ and $T \in \lbrace \overline{T}, \underline{T} \rbrace$.

\medskip
\noindent
\textbf{2.} Let $q_{i}$ be odd and $r_{i}$ even, then it holds:
\begin{equation}
\label{Ineq_40}
\hspace*{-0.1 cm}
\begin{array}{rcl}
s_{i}(x)
           & \!\!\!\! = \!\!\!\!&
\alpha_{i}x^{p_{i}}\!\cos^{q_{i}}\!\!x\sin^{r_{i}}\!\!x
                                                                                 \\ [1.5ex]
           & \!\!\!\! = \!\!\!\!&
\frac{\alpha_{i}x^{p_{i}}}{2^{q_{i}+r_{i}-1}} \!\!
\displaystyle\sum\limits_{k=0}^{\euft{m}_{i}}(-1)^{\frac{r_{i}}{2}+k} \!
\displaystyle\sum\limits_{r=0}^{k}(-1)^{r}\NBC{\!\!q_{i}\!\!}{\!\!r\!\!}\NBC{\!\!r_{i}\!\!}{\!\!k-r\!\!}
\!\cos\!{\big(}\!(\!q_{i}\!+\!r_{i}\!-\!2k)x\!{\big)}
                                                                                  \\[1.5ex]
           & \!\!\!\! = \!\!\!\!&  \frac{x^{p_{i}}}{2^{q_{i}+r_{i}-1}}\!\!\displaystyle\sum\limits_{k=0}^{\euft{m}_{i}}
{\Bigg(} \!
\displaystyle\sum\limits_{r=0}^{k}\!\!\alpha_{i}(-1)^{\frac{r_{i}}{2}+k+r}
\NBC{\!\!q_{i}\!\!}{\!\!r\!\!}\NBC{\!\!r_{i}\!\!}{\!\!k-r\!\!}  \!\!
{\Bigg)}
\!\cos\!{\big(}\!(\!q_{i}\!+\!r_{i}\!-\!2k\!)x\!{\big)}.
\end{array}
\end{equation}
Let us mark with
$\gamma_{k}=\displaystyle\sum\limits_{r=0}^{k}\alpha_{i}(-1)^{\frac{r_{i}}{2}+k+r}
\NBC{\!\!q_{i}\!\!}{\!\!r\!\!}\NBC{\!\!r_{i}\!\!}{\!\!k-r\!\!}$. Then, for every sub-addend
$\gamma_{k}\cos\!{\big(}(q_{i}+r_{i}-2k)x{\big)}$, depending on the sign of $\gamma_{k}$, two cases are possible:

\medskip
\noindent
\quad
1) if $\gamma_{k}>0$:
\begin{equation}
\label{Ineq_41}
\gamma_{k}\cos{\big(}(q_{i}+r_{i}-2k)x{\big)}>
\gamma_{k}\underline{T}^{\cos,0}_{\,4l^{(i)}_{k}+2}{\big(}(q_{i}+r_{i}-2k)x{\big)},
\end{equation}

\smallskip
\noindent
\quad
2) if $\gamma_{k}<0$:
\begin{equation}
\label{Ineq_42}
\gamma_{k}\cos{\big(}(q_{i}+r_{i}-2k)x{\big)}>
\gamma_{k}\overline{T}^{\,\cos,0}_{4l^{(i)}_{k}+0}{\big(}(q_{i}+r_{i}-2k)x{\big)}.
\end{equation}

\smallskip
\noindent
Let the addend $s_{i}(x)$ be written in the form:
\begin{equation}
\label{Ineq_43}
s_{i}(x)=\frac{x^{p_{i}}}{2^{q_{i}+r_{i}-1}}
\displaystyle\sum\limits_{k=0}^{\euft{m}_{i}}\gamma_{k}\cos{\big(}(q_{i}+r_{i}-2k)x{\big)}.
\end{equation}
Then it holds:
\vspace{-0.3 cm}
\begin{equation}
\label{Ineq_44}
s_{i}(x)
>
\mbox{\large $\tau$}_{i}(x) = \frac{x^{p_{i}}}{2^{q_{i}+r_{i}-1}}
\displaystyle\sum\limits_{k=0}^{\euft{m}_{i}}
\gamma_{k}T^{\cos,0}_{4l^{(i)}_{k}+v}{\big(}(q_{i}+r_{i}-2k)x{\big)},
\end{equation}
where
$
v=\left\{
\begin{array}{ll}
2, & \gamma_{k}>0, \\ [1.5ex]
0, & \gamma_{k}<0
\end{array}
\right.
$, $l^{(i)}_{k} \in \mathbb{N}_{0}$ and $T \in \lbrace \overline{T}, \underline{T} \rbrace$.

\smallskip
\noindent
\textbf{3.} Let $q_{i}$ and $r_{i}$ be even, then based on the previous case (under {\bf 2.}) it holds:
\begin{equation}
\label{Ineq_45}
\hspace*{-1.75 mm}
\begin{array}{rcl}
s_{i}(x)
           & \!\!\!\! = \!\!\!\!&
\frac{x^{p_{i}}}{2^{q_{i}+r_{i}-1}} \!\!
{\Bigg(}\!
\displaystyle\sum\limits_{k=0}^{\euft{m}_{i}}
{\Bigg(}\!
\displaystyle\sum\limits_{r=0}^{k}\alpha_{i}(-1)^{\frac{r_{i}}{2}+k+r}\NBC{\!\!q_{i}\!\!}{\!\!r\!\!}\NBC{\!\!r_{i}\!\!}{\!\!k-r\!\!}\!\!
{\Bigg)}\!
\!\cos\!{\big(}\!(q_{i}\!+\!r_{i}\!-\!2k)x\!{\big)}
                                                                                         \\[1.5ex]
           & \!\!\!\! + \!\!\!\!&
\frac{1}{2}(-1)^{\frac{2r_{i}+q_{i}}{2}}
\displaystyle\sum\limits_{r=0}^{\frac{q_{i}+r_{i}}{2}}(-1)^{r}
\NBC{\!\!q_{i}\!\!}{\!\!r\!\!}\nbc{\!\!r_{i}\!\!}{\!\frac{q_{i}+r_{i}}{2}\!-\!r\!}\!\!
{\Bigg)}
                                                                                         \\ [1.5ex]
           & \!\!\!\! > \!\!\!\!&
\mbox{\large $\tau$}_{i}(x) = \frac{x^{p_{i}}}{2^{q_{i}+r_{i}-1}}
{\Bigg(}\!
\displaystyle\sum\limits_{k=0}^{\euft{m}_{i}}
\gamma_{k} T^{\cos,0}_{4l^{(i)}_{k}+v}{\big(}(q_{i}+r_{i}-2k)x{\big)}
                                                                                         \\ [1.5ex]
           & \!\!\!\! + \!\!\!\!&
\frac{1}{2}(-1)^{\frac{2r_{i}+q_{i}}{2}}
\displaystyle\sum\limits_{r=0}^{\frac{q_{i}+r_{i}}{2}}(-1)^{r}
\NBC{\!\!q_{i}\!\!}{\!\!r\!\!}\nbc{\!\!r_{i}\!\!}{\!\!\frac{q_{i}+r_{i}}{2}-r\!\!}\!\!
{\Bigg)},
\end{array}
\end{equation}

\vspace*{-1.0 mm}

\noindent
where
$
v=\left\{
\begin{array}{ll}
2, & \gamma_{k}>0, \\ [1.5ex]
0, & \gamma_{k}<0
\end{array}
\right.
$, $l^{(i)}_{k} \in \mathbb{N}_{0}$ and $T \in \lbrace \overline{T}, \underline{T} \rbrace$.

\medskip
\noindent
\qquad
Comparing all the addends $s_{i}(x)$ $(i\!=\!1,2,\ldots,n)$ that appear in the sum (\ref{Ineq_1}),
according to the above stated cases, we get the polynomial
\begin{equation}
\label{Ineq_46}
P(x) = \sum_{i=1}^{n}{\mbox{\large $\tau$}_{i}(x)}
\end{equation}
\vspace{-0.15 cm}
{\big(}$\,$downward approximation of the function $f(x)$ in (\ref{Ineq_1})$\,${\big)}; i.e. it holds:
\begin{equation}
\label{Ineq_47}
f(x)> P(x).
\end{equation}
\vspace{-0.1 cm}
On the basis of the previous consideration, the following statement ensues:
\begin{theorem}
Let the following properties of the polynomial $P(x) \!=\! \displaystyle\sum_{i=1}^{n}{\mbox{\large $\tau$}_{i}(x)}$
\hfill

\vspace*{-0.25 cm}

\noindent
be true:

\medskip
\noindent$\;\;\;\;\;\;\;$
\begin{minipage}[b]{120.00 mm}
$\,(i)\,$ there is at least one~positive~real~root of the polynomial $P(x)\,$;
\end{minipage}

\medskip
\noindent$\;\;\;\;\;\;\;$
\begin{minipage}[b]{120.00 mm}
$(ii)$ $P(x)\!>\!0$ for $x\!\in\!(0,x^{\myast})$, where $x^{\myast}$
is the least positive~real~root \hfill

\hspace*{5.25 mm} of the polynomial $P(x)$;
\end{minipage}

\smallskip
\noindent
then it is valid
$$
f(x)>0
$$
for $x \in (0,x^{\myast})\subseteq (0,\delta)$.
\end{theorem}
\begin{remark}
Let us notice that hereby the proof of the inequality $f(x)\!>\!0$ has been obtained for
$x\!\in\!(0,\delta_{2})$, where $\delta_{2}\!=\!x^{\myast}$. The previous Theorem may be
applied in the interval $(\delta_{1}, 0)$ by introducing the substitute $t\!=\!-x$.
\end{remark}
\begin{remark}
If there is not at least one positive real root of the polynomial $P(x)$ and  $P(x)\!>\!0$ for $x \!\in\! (0,\infty)$, then it is valid
$f(x)\!>\!0$ for $x \!\in\! (0,\infty)$.
\end{remark}
The previous Theorem determines a method of proving a class of trigonometric inequalities
based on approximations of the sine and cosine functions by Maclaurin polynomials.

\medskip 
\noindent
\textbf{II}\enskip
We will consider completeness of the given method for the function $f(x)$, of the mixed trigonometric polynomial,
which is not a classical polynomial. Let us start from the following auxiliary statement.

\begin{lemma}
Let $f \!:\! (\delta_1, \delta_2) \!\longrightarrow\! \mathbb{R}$, $\delta_1 \!\leq\! 0 \!\leq\! \delta_2$
and $\delta_1\!<\!\delta_2$, be a real, non-constant, analytic function such that domain $(\delta_1, \delta_2)$ belongs to the interval of convergence of the function $f(x)$.

\smallskip
\noindent
If $f(0)\!\neq\!0$, then it holds:

\medskip
\noindent
{\bf 1.}
\begin{equation}
\label{eq_24}
f(0) > 0
\Longleftrightarrow
{\big (}\exists \, x^{\mbox{\tiny $\myplus$}} \!\in\! (0,\delta_2]{\big )} {\big (}\forall \, x \!\in\! (0, x^{\mbox{\tiny $\myplus$}}){\big )}
f(x) > 0,
\end{equation}

\noindent
{\bf 2.}
\begin{equation}
\label{eq_25}
f(0) < 0
\Longleftrightarrow
{\big (}\exists \, x^{\mbox{\tiny $\myplus$}} \!\in\! (0,\delta_2]{\big )} {\big (}\forall \, x \!\in\! (0, x^{\mbox{\tiny $\myplus$}}){\big )}
f(x) < 0.
\end{equation}

\smallskip
\noindent
If $f(0) \!= \ldots =\! f^{(n-1)}(0) = 0 \;\wedge\; f^{(n)}(0) \neq 0$, for some $n \in \mathbb{N}$, then it holds:

\medskip
\noindent
{\bf 3.}
\begin{equation}
\label{eq_26}
f^{(n)}(0) > 0
\Longleftrightarrow
{\big (}\exists \, x^{\mbox{\tiny $\myplus$}} \!\in\! (0,\delta_2]{\big )} {\big (}\forall \, x \!\in\! (0, x^{\mbox{\tiny $\myplus$}}){\big )}
f(x) > 0,
\end{equation}

\noindent
{\bf 4.}
\begin{equation}
\label{eq_27}
f^{(n)}(0) < 0
\Longleftrightarrow
{\big (}\exists \, x^{\mbox{\tiny $\myplus$}} \!\in\! (0,\delta_2]{\big )} {\big (}\forall \, x \!\in\! (0, x^{\mbox{\tiny $\myplus$}}){\big )}
f(x) < 0.
\end{equation}

\end{lemma}
\begin{proof}
Let $f(x)$ be a non-constant function with Maclaurin series expansion
\begin{equation}
\label{eq_28}
f(h) = f(0) + \frac{f'(0)}{1!}h + \frac{f''(0)}{2!}h^2 + \ldots, \qquad (h > 0).
\end{equation}
In the proof we use the method from \cite{Godement_2004} (pages 157, 158), by which
it has been shown that zeros of non-constant analytic function are isolated.

\smallskip
\noindent
{\bf 1.} {\boldmath $(\Rightarrow)$} Let $f(0) \!>\! 0$. Let us note (\ref{eq_28})
in the form
\begin{equation}
\label{eq_29}
f(h) = f(0){\big (}1 + g(h){\big )},
\end{equation}
where $g(h)$ is the real analytical function. Then there exist $x^{\mbox{\tiny $\myplus$}} > 0$
and $M > 0$ such that $|g(h)| < Mh$ and $M h < 1/2$ for every
$h \in (0, x^{\mbox{\tiny $\myplus$}})$. Hence, we conclude that $f(h) = f(0) + f(0) \, g(h) > f(0) - f(0) Mh
> f(0) / 2 > 0$ for $h \in (0, x^{\mbox{\tiny $\myplus$}})$.
\mbox{\boldmath $(\Leftarrow)$}~Let us suppose that~there exists $x^{\mbox{\tiny $\myplus$}} \!\in\! (0,\delta_2]$ such that
for every $x \!\in\! (0, x^{\mbox{\tiny $\myplus$}})$ it holds $f(x) \!>\! 0$. Consequently, it ensues that $f(x)$
is a positive function in arbitrarily small right-hand neighbourhood of the point $x\!=\!0$.
Let $g(x)$ be~the~function considered in the previous part of the proof. Then there exist
$M > 0$ and $x_{1} \!\in\! (0, x^{\mbox{\tiny $\myplus$}}]$
such that for every $x \!\in\! (0, x_{1})$ it holds $|g(x)| < Mx < 1/2$. If it holds
$f(0) \!<\! 0$, then for $x \!\in\! (0, x_1) \subseteq (0, x^{\mbox{\tiny $\myplus$}})$
we have the contradiction $f(x)\!=\!f(0){\big (}1 + g(x){\big )}\!<\! 0$.
Hereby it has been proved $f(0) \!>\! 0$.

\smallskip
\noindent
{\bf 2.} It is sufficient to consider function $-f(x)$ instead of function $f(x)$ and to apply {\bf 1.}

\smallskip
In the case $f(0) \!= \ldots =\! f^{(n-1)}(0) = 0 \;\wedge\; f^{(n)}(0) \!\neq\! 0$,
for some $n \in \mathbb{N}$, the point $x=0$ has been isolated zero of order $n$.
Let us note further (\ref{eq_28})
in the form
\begin{equation}
\label{eq_30}
f(h)
=
\frac{f^{(n)}(0)}{n!}h^n
+
\frac{f^{(n+1)}(0)}{(n+1)!}h^{n+1}
+
\ldots,
\qquad (h > 0).
\end{equation}

\smallskip
\noindent
{\bf 3.} {\boldmath $(\Rightarrow)$} Let $f^{(n)}(0) \!>\! 0$. Let us note (\ref{eq_30}) in the form
\begin{equation}
\label{eq_31}
f(h) = \frac{f^{(n)}(0)}{n!} \, h^{n}{\big (}1 + g(h){\big )},
\end{equation}
where $g(h)$ is the real analytical function. Then there exist $x^{\mbox{\tiny $\myplus$}} > 0$
and \mbox{$M > 0$} such that $|g(h)|\!<\!Mh$ and $Mh\!<\!1/2$  for every
$h \in (0, x^{\mbox{\tiny $\myplus$}})$. Hence, we conclude that
$f(h)\!=\!\frac{f^{(n)}(0)}{n!}\,h^n\!+\!\frac{f^{(n)}(0)}{n!}\,h^n g(h)\!>\!\frac{f^{(n)}(0)}{n!}\,h^n\!-\!\frac{f^{(n)}(0)}{n!}\,h^n Mh\!>\!\frac{1}{2}\frac{f^{(n)}(0)}{n!}\,h^n\!>\!0$
for~$h \in (0, x^{\mbox{\tiny $\myplus$}})$.
\mbox{\boldmath $(\Leftarrow)$}~Let us suppose that~there exists $x^{\mbox{\tiny $\myplus$}} \!\in\! (0,\delta_2]$ such that
for every $x \!\in\! (0, x^{\mbox{\tiny $\myplus$}})$ it holds $f(x) \!>\! 0$. Consequently, it ensues that $f(x)$
is a positive function in arbitrarily small right-hand neighbourhood of the point $x\!=\!0$. Let $g(x)$ be
the function considered in the previous part of the proof. Then there exist $M\!>\!0$ and $x_{1} \!\in\! (0, x^{\mbox{\tiny $\myplus$}}]$
such that for every $x \!\in\! (0, x_{1})$ it holds $|g(x)|\!<\!Mx\!<\!1/2$. If it holds
$f^{(n)}(0)\!<\!0$, then for $x\!\in\!(0, x_1)\!\subseteq\!(0, x^{\mbox{\tiny $\myplus$}})$~we~have
the~contradiction $f(x)\!=\!\frac{f^{(n)}(0)}{n!}\,x^{n}{\big (}1\!+\!g(x){\big )}\!<\!0$.
Hereby it has~been~proved~\mbox{$f^{(n)}(0)\!>\!0$}.

\smallskip
\noindent
{\bf 4.}  It is sufficient to consider function $-f(x)$ instead of function $f(x)$ and to apply~{\bf 3.}
\end{proof}
Based on the previous statement it follows:
\begin{theorem}
Let $f \!:\! (\delta_1, \delta_2) \!\longrightarrow\! \mathbb{R}$, $\delta_1 \!\leq\! 0 \!\leq\! \delta_2$
and $\delta_1\!<\!\delta_2$, be a real, non-constant, analytic function such that domain $(\delta_1, \delta_2)$
belongs to the interval of convergence of the function $f(x)$.$\,$Then the equivalences
\begin{equation}
\label{eq_32}
\begin{array}{lc}
                            &
{\big (}\exists \, x^{\mbox{\tiny $\myplus$}} \!\in\! (0,\delta_2]{\big )} {\big (}\forall \, x \!\in\! (0, x^{\mbox{\tiny $\myplus$}}){\big )} f(x) > 0
                                                                                                              \\[1.0 ex]
\!\!\!\!\!\!\!\!
\Longleftrightarrow         &                                                                                 \\[1.5 ex]
                            &
f(0) > 0
\;\;\vee\;\;
{\Big (}{\big (}\exists \, n \!\in\! \mathbb{N}{\big )} f(0) \!= \ldots =\! f^{(n-1)}(0) = 0 \;\wedge\; f^{(n)}(0) > 0{\Big )}
\end{array}
\end{equation}
or
\begin{equation}
\label{eq_33}
\begin{array}{lc}
                            &
{\big (}\exists \, x^{\mbox{\tiny $\myplus$}} \!\in\! (0,\delta_2]{\big )} {\big (}\forall \, x \!\in\! (0, x^{\mbox{\tiny $\myplus$}}){\big )} f(x) < 0
                                                                                                              \\[1.0 ex]
\!\!\!\!\!\!\!\!
\Longleftrightarrow         &                                                                                 \\[1.5 ex]
                            &
f(0) < 0
\;\;\vee\;\;
{\Big (}{\big (}\exists \, n \!\in\! \mathbb{N}{\big )} f(0) \!= \ldots =\! f^{(n-1)}(0) = 0 \;\wedge\; f^{(n)}(0) < 0{\Big )}
\end{array}
\end{equation}

\smallskip
\noindent
are true.
\end{theorem}

\break

\noindent
\begin{remark}
In the following consideration we observe that $f(x)$ is a mixed trigonometric polynomial which is not a classical polynomial. Such functions
are analytic with the interval of convergence which is determined as a set of real numbers. $\!\!\!\!$ That is why the problem~whether there
is an interval $(0, x^{\mbox{\tiny $\myplus$}})$ for the mixed trigonometric polynomial $f(x)$, for some \mbox{$x^{\mbox{\tiny $\myplus$}} \!>\! 0$},
in which $f(x)$ is of the constant sign, represents a decidable problem
based on the equivalences $(\ref{eq_32})$ and $(\ref{eq_33})$.
\end{remark}
We will consider completeness of the given method for the function $f(x)$, of the mixed trigonometric polynomial,
which is not a classical polynomial under the assumption
\begin{equation}
\label{Ineq_48}
{\big (}\exists  \, x^{\mbox{\tiny $\myplus$}} \!\in\! (0, \delta){\big )}
{\big (}\forall x \!\in\! (0, x^{\mbox{\tiny $\myplus$}}){\big )}
f(x) \!>\! 0.
\end{equation}
We will show that for the function $f(x)$ in every sub-interval
$\left(a, b\right) \!\subset\! \left(0, x^{\mbox{\tiny $\myplus$}}\right)$, where
\mbox{$0 \!<\! a \!<\! b \!<\! x^{\mbox{\tiny $\myplus$}}$}, there exists a positive downward
polynomial approximation $P(x)$. Let all the indexes
$l_{k}^{(i)}$ {\big (}\mbox{$i \!\in\! \{1,\ldots,n\}$}
and \mbox{$k \in \{0,\ldots,\eu{m}_{i}\}$}{\big )} have the same value $l_{k}^{(i)} \!=\! \mbox{\small $K$} \!\in\! N_{0}$.
As a function of index $\mbox{\small $K$}$, a polynomial
$
P^{\mbox{\tiny $\left[ K \right]$}}(x)
=
\sum_{i=1}^{n}{\mbox{\large $\tau$}^{\mbox{\tiny $\left[ K \right]$}}_{i}(x)}
$ is formed.
Previously formed polynomial $P^{\mbox{\tiny $\left[ K \right]$}}(x)$ of index $\mbox{\small $K$}$
is downward polynomial approximation of the function $f(x)$ such that it is valid
\begin{equation}
\label{Ineq_49}
\lim\limits_{K \longrightarrow \infty} P^{\mbox{\tiny $\left[ K \right]$}}(x) = f(x),
\end{equation}
where the previous convergence is uniform in $[\,0, x^{\mbox{\tiny $\myplus$}}]$.
The uniform property of the convergence follows based on the fact that for every $i=1,\ldots,n$
the convergence
\begin{equation}
\label{Ineq_50}
\lim\limits_{K \longrightarrow \infty} \mbox{\large $\tau$}^{\mbox{\tiny $\left[ K \right]$}}_{i}(x)
=
s_{i}(x),
\end{equation}
is uniform in $[\,0, x^{\mbox{\tiny $\myplus$}}]$.
Based on \mbox{$P^{\mbox{\tiny $\left[ K \right]$}}(x) < f(x)$} and
\mbox{$P^{\mbox{\tiny $\left[ K \right]$}}(x) \rightrightarrows f(x)$},
in $[\,0, x^{\mbox{\tiny $\myplus$}}]$, we get the following statement about the completeness of the discussed method.
\begin{theorem}
Let for the mixed trigonometric polynomial $f(x)$ which is not a classical polynomial, the condition $(\ref{Ineq_48})$ is valid.
Then in every interval $\left(a, b\right) \!\subset\! \left(0, x^{\mbox{\tiny $\myplus$}}\right)$, where \mbox{$0 \!<\! a \!<\! b \!<\! x^{\mbox{\tiny $\myplus$}}$},
there exists downward polynomial approximation $P^{\mbox{\tiny $\left[ K \right]$}}(x)$ of the index
$\mbox{\small $K$}$ such that
\begin{equation}
\label{Ineq_51}
{\big (}\forall x \!\in\! (a,b){\big )} \, f(x) > P^{\mbox{\tiny $\left[ K \right]$}}(x) > 0.
\end{equation}
\end{theorem}
\begin{remark}
Under the assumptions of the previous Theorem for the function $f(x)$ it follows the completeness
of the method in the sense that it is possible in every interval $(a,b) \subset (0,x^{\mbox{\tiny $\myplus$}})$,
where \mbox{$0 \!<\! a \!<\! b \!<\! x^{\mbox{\tiny $\myplus$}}$}, to prove the inequality
$f(x) \!>\! 0$ by using some downward approximation $P^{\mbox{\tiny $\left[ K \right]$}}(x)$.
\end{remark}

\noindent
\textbf{2.1. Improving of the method}

\medskip
\noindent
\qquad
Let us emphasise that the previous method may be applied to the functions of the form
$f(x)=\sum_{i=1}^{n}\alpha_{i}h_{i}(x)\!\cos^{q_{i}}\!\!x\sin^{r_{i}}\!\!x$
for $x \in (0,\delta)$, where $h_{i}(x)$ is a polynomial, in such a way that two
possibilities exist. The first possibility is when the polynomial $h_{i}(x)$ is of
the constant sign in the given interval and then we can see the cases $h_{i}(x)>0$ or $h_{i}(x)<0$,
and we do that by analogy with the previously described procedure. On the other hand,
we have a possibility that the polynomial $h_{i}(x)$ is not of the constant sign.
Then, $\alpha_{i}h_{i}(x)\!\cos^{q_{i}}\!\!x\sin^{r_{i}}\!\!x$ may be written as
a sum of addends of the form $s_{i}(x)$, and then we can apply the previously
described method for each of those addends.

\bigskip
\noindent
\textbf{2.2. End of the procedure}

\medskip
\noindent
\qquad
Let the indexes $l_{k}^{(i)}\,\big{(}i \in \lbrace1,\ldots,n\rbrace$ and $k \in \lbrace0,\ldots,\eu{m}_{i}\rbrace\big{)}$, which appear in the polynomial $P(x)$, be aligned: $l_{0}, l_{1},\ldots,l_{m}$; where $m\!+\!1$
is the overall number of sub-addends which come from every addend $s_{i}(x)$. The indexes
$l_{0}, l_{1}, \ldots l_{m}$ have been determined in (\ref{Ineq_39}), (\ref{Ineq_44}) and (\ref{Ineq_45}). Let us notice that according to the index
$l_{s}$ it holds:
\begin{equation}
\label{Ineq_52}
f(x) > P(x, l_{0}, l_{1}, \ldots, l_{s}+1, \ldots, l_{m}) > P(x, l_{0}, l_{1}, \ldots, l_{s},\ldots, l_{m})
\end{equation}
for every index $s \in \lbrace 0,1,2, \ldots, m \rbrace$  and $l_{s} \in \mathbb{N}_{0}$.
It should be noted that the interval in which the sharp inequality
\begin{equation}
\label{Ineq_53}
P(x, l_{0}, l_{1}, \ldots, l_{s}+1, \ldots, l_{m}) > P(x, l_{0}, l_{1}, \ldots, l_{s},\ldots, l_{m})
\end{equation}
is valid, may be determined according to the Lemmas 1.1. and 1.2.
By increasing every index $l_{s}$, the intervals of validity of (\ref{Ineq_52}) are expanded based on the Lemmas 1.1 and 1.2.
and we get even better and better downward approximations of the function $f(x)$.
The previously described method defines a procedure which ends when at least one $(m\!+\!1)$-tuple of the indexes
$(l_{0}, l_{1}, \ldots, l_{m}) = (\hat{l}_{0}, \hat{l}_{1}, \ldots, \hat{l}_{m})$ has been determined for which it is valid:
\begin{equation}
\label{Ineq_54}
P(x, \hat{l}_{0}, \hat{l}_{1},\ldots, \hat{l}_{m})>0
\end{equation}
for $x \!\in\! (0,\delta)$. By completing the procedure, we get a proof of the initial inequality (\ref{Ineq_1}).
\begin{remark}
This method represents a generalisation of the method that C.~Mortici used for proving inequalities in the article
{\rm \cite{Mortici_2011}}. The method comes down to proving polynomial inequalities of the form $P(x)\!>\!0$ for
$x \!\in\! (0,\delta)$ which is a decidable problem according to the results by Tarski {\rm \cite{Poonen_14}}.
\end{remark}

\noindent
By using this method it is our aim in this article to get some well-known results concerning
the inequalities of the form (\ref{Ineq_1}) that have been considered in the lately published articles.

\section{The applications of the method} 

\smallskip
\noindent
\qquad
In this section we consider the applications of the method based on the Theorem 2.1. in some concrete inequalities.

\medskip
\noindent
\textbf{3.1. A proof of an inequality from the article [6]} 

\noindent
In the article \cite{Chen_Cheung_2012} C.$-$P. Chen and W.$-$S. Cheung have lately proved the following statement (Theorem 2 in \cite{Chen_Cheung_2012}):
\begin{theorem}
{\rm \textbf{(i)}}
For $0<x<\pi/2$, we have
\begin{equation}
\label{Ineq_55}
\Big{(}\frac{x}{\sin x}\Big{)}^{\!2}\!+\frac{x}{\tan x}<2+\frac{2}{45}x^{3}\tan x.
\end{equation}

\noindent
The constant $\frac{2}{45}$ is the best possible.

\smallskip
\noindent
{\rm \textbf{(ii)}}
For $0<x<\pi/2$, we have
\begin{equation}
\label{Ineq_56}
\Big{(}\frac{x}{\sin x}\Big{)}^{\!2}\!+\frac{x}{\tan x}<2+\frac{2}{45}x^{4}+\frac{8}{945}x^{5}\tan x.
\end{equation}

\noindent
The constant $\frac{8}{945}$ is the best possible.
\end{theorem}

\noindent
Now we are presenting a proof of an inequality (\ref{Ineq_55}).
\begin{proof}
The requested inequality is equivalent to the inequality $f(x)\!>\!0$ for $x \!\in \!(0,\pi/2)$,
where
\begin{equation}
\label{Ineq_57}
\begin{array}{rcl}
f(x)=2\cos x\sin^{2}\!x+\frac{2}{45}x^{3}\sin^{3}\!x-x \cos^{2}\!x\sin x-x^{2}\cos x
\end{array}
\end{equation}

\smallskip
\noindent
is one concrete mixed trigonometric polynomial.
Let us notice that $x=0$ is zero of the eighth order of the function $f(x)$.
According to the Theorem 1.5. the function $f(x)$ may be written in the following way:
\begin{equation}
\label{Ineq_58}
f(x)=\mbox{\small $\frac{1}{2}$}\cos x-x^{2}\cos x
-\mbox{$\small\frac{1}{2}$}\cos3x-(\underbrace{\mbox{$\small\frac{1}{90}$}x^{3}
+\mbox{$\small\frac{1}{4}x$}}_{(>0)})\sin3x+(\underbrace{\mbox{$\small\frac{1}{30}$}x^{3}
-\mbox{$\small\frac{1}{4}x$}}_{(<0)})\sin x.
\end{equation}

\vspace*{-1.0 mm}

\noindent
Then, according to the Lemmas 1.1. and 1.2. and the description of the method, the following inequalities:
$\cos y > \underline{T}_{\,k}^{\cos,0}(y) \,\, (k=6)$, $\cos y < \overline{T}_{k}^{\,\cos,0}(y) \,\, (k=12)$ and
$\sin y < \overline{T}_{k}^{\,\sin,0}(y) \,\, (k=13)$ are true, for
$y \in \big{(}0,\sqrt{(k+3)(k+4)}\,\big{)}$.

\smallskip
\noindent
For $x \in (0,\pi/2)$ it is valid:
\begin{equation}
\hspace*{-0.4cm}
\label{Ineq_59}
\begin{array}{rcl}
f(x)
\!&\!\!>\!\!&\!
\mbox{\small $\frac{1}{2}$}\underline{T}_{\,6}^{\cos,0}(x)-x^{2}\overline{T}_{12}^{\,\cos,0}(x)
-\mbox{$\small\frac{1}{2}$}\overline{T}_{12}^{\,\cos,0}(3x)-(\underbrace{\mbox{$\small\frac{1}{90}$}x^{3}
+\mbox{$\small\frac{1}{4}x$}}_{(>0)})\overline{T}_{13}^{\,\sin,0}(3x)
                                                                                              \\[1.5 ex]
\!&\!\!+\!\!&\!
(\underbrace{\mbox{$\small\frac{1}{30}$}x^{3}-
\mbox{$\small\frac{1}{4}x$}}_{(<0)})\overline{T}_{13}^{\,\sin,0}(x)
=
P_{16}(x),
\end{array}
\end{equation}
where $P_{16}(x)$ is the polynomial
\begin{equation}
\label{Ineq_60}
\begin{array}{rcl}
P_{16}(x)
\!&\!\!=\!\!&\!
\mbox{\footnotesize $\displaystyle\frac{x^{8}}{186810624000}$}
{\big (}
\mbox{\small $-531440x^{8}\!-\!2746332x^{6}\!-\!8885955x^{4}\!$}
                                                                                  \\[2.0 ex]
\!&\!\!-\!\!&\!\!\mbox{\small $118584180x^{2}\!+\!1183782600$}{\big )}
                                                                                  \\[1.0 ex]
\!&\!\!=\!\!&\!
\mbox{\footnotesize $\displaystyle\frac{x^{8}}{186810624000}$}
P_{8}(x).
\end{array}
\end{equation}
Then we determine the sign of the polynomial $P_{8}(x)$ for $x \in (0,\pi/2)$.
By introducing the substitute $z=x^{2}$, we get the fourth degree polynomial:
\begin{equation}
\label{Ineq_61}
P_{4}(z)=\mbox{\small $-531440\,z^{4}
\!-\!2746332\,z^{3}
\!-\!8885955\,z^{2}
\!-\!118584180\,z
\!+\!1183782600$}.
\end{equation}

\noindent
A real numerical factorization of the polynomial $P_{4}(z)$, has been determined via {\sf Matlab}
software, and given with
\begin{equation}
\label{Ineq_62}
P_{4}(z)
\!=\!
\alpha(z\!-\!z_{1})(z-z_{2})(z^2\!+\!pz\!+\!q),
\end{equation}
where
$
\alpha\!=\!-531440,
z_{1}\!=\!4.503628...,
z_{2}\!=\!-9.049...,
p
\!=\!0.621...,
q
\!=\!54.652...;
$
whereby the inequality \mbox{$p^{\,2}\!-4q\!<\!0$} is true. The polynomial $P_{4}(z)$ has exactly
two simple real roots with a symbolic radical representation and the corresponding numerical values $z_{1}$ and $z_{2}$.
Since $P_{4}(0)>0$ it follows that $P_{4}(z)>0$ for the values $z \in (0,z_{1})\subset(z_{2},z_{1})$ .
Finally, we conclude that
\begin{equation}
\hspace*{-4.25 mm}
\begin{array}{rcl}
\label{Ineq_63}
P_{8}(x) \!>\! 0 \;\mbox{\rm for}\; x \!\in\! (0,\!\sqrt{\!z_{1}\!})\!=\!(0,\mbox{\small $\!2.122...$})
& \!\!\!\!\!\Longrightarrow\!\!\!\!\! &
P_{16}(x)\! >\! 0 \;\mbox{\rm for}\; x \!\in\! (0,\mbox{\small $\!2.122...$})
                                                                                        \\[1.0 ex]
& \!\!\!\!\!\Longrightarrow\!\!\!\!\! &
f(x)\!>\! 0 \;\mbox{\rm for}\; x \!\in\! (0,\!\mbox{\small $\displaystyle\frac{\pi}{2}$})\!\subset
\!(0,\mbox{\small $\!2.122...$}).\!\!
\end{array}
\end{equation}

\smallskip
\noindent
Let us notice that the least positive real root of the downward approximation of the function $f(x)$,
i.e. of the polynomial $P_{16}(x)$, is \mbox{$x^{\ast}\!=\!\sqrt{z_{1}}\!=\!2.122175...\!>\!\pi/2$}.
Elementary calculus gives that the constant $\frac{2}{45}$ is the best possible.
\end{proof}

\medskip
\noindent
\textbf{3.2. A proof of an inequality from the paper [35]} 

\smallskip
\noindent
In the paper \cite{Sun_Zhu_2011} Z.$-$J. Sun and L. Zhu have posed an open problem,
to prove the following statement:
\begin{theorem}
Let $0<x<\pi/2$. Then
\begin{equation}
\begin{array}{rcl}
\label{Ineq_64}
\hspace{-0.3 cm}
\mbox{\small $\displaystyle\frac{(2\pi^{4}/3)x^{3}+(8\pi^{4}/15-16\pi^{2}/3)x^{5}}{(\pi^{2}-4x^{2})^{2}}$}
& \!\!\!<\!\!\! &
\mbox{\small $x\sec^{2} \!x-\tan x$}
                                                                                            \\[2.5 ex]
& \!\!\!<\!\!\! &
\mbox{\small $\displaystyle\frac{(2\pi^{4}/3)x^{3}+(256/\pi^{2}-8\pi^{2}/3)x^{5}}{(\pi^{2}-4x^{2})^{2}}$},
\end{array}
\end{equation} hold, where $(8\pi^{4}/15-16\pi^{2}/3)$ and $(256/\pi^{2}-8\pi^{2}/3)$
are~the~best~constants~in~$(\ref{Ineq_64})$.
\end{theorem}

\noindent
Now we are presenting a proof of the previous statement.

\begin{proof}

\smallskip
\noindent
\textbf{I}
We prove the inequality:
\begin{equation}
\label{Ineq_65}
\frac{(2\pi^{4}/3)x^{3}+(8\pi^{4}/15-16\pi^{2}/3)x^{5}}{(\pi^{2}-4x^{2})^{2}}
< x\sec^{2} \!x-\tan x
\end{equation}
for $x \in (0,\pi/2)$. The requested inequality is equivalent to the inequality
$f(x)\!>\!0$ for $x \!\in (0,\pi/2)$, where
\begin{equation}
\label{Ineq_66}
\begin{array}{rcl}
f(x)=\mbox{\small $x(\pi^{2}\!-\!4x^{2})^{2}$}\!\!-\!\!\mbox{\small $(\pi^{2}\!-\!4x^{2})^{2}$}
\cos x\sin x\!-\!\big{(}
& \!\!\!\!\!\!(\!\!\!\!\!\! &\mbox{\small $2\pi^{4}/3)x^{3}$}+
                                                                                              \\[1.5 ex]
& \!\!\!\!\!\!(\!\!\!\!\!\! &\mbox{\small $8\pi^{4}\!/15\!-\!16\pi^{2}\!/3$})\mbox{\small $x^{5}$}\big{)}\!\cos^{2}x
\end{array}
\end{equation}

\smallskip
\noindent
is one concrete mixed trigonometric polynomial.
Let us notice that $x=0$ is zero of the seventh order and $x=\pi/2$ is zero of the second order
of the function $f(x)$. Let us consider two cases:

\smallskip
\noindent
\textbf{1)}
If $x \in (0,1.136)$:

\smallskip
\noindent
According to the Theorem 1.5. the function $f(x)$ may be written in the following way:
\begin{equation}
\label{Ineq_67}
\begin{array}{rcl}
f(x)
& \!\!\!\!\!=\!\!\!\!\! &
\mbox{\small $x(\pi^{2}\!-\!4x^{2})^{2}$}\!\!-\!\!
\mbox{\small $\displaystyle\frac{(\pi^{2}\!-\!4x^{2})^{2}}{2}$}\sin 2x\!-\!\big{(}
\mbox{\small $(\pi^{4}/3)x^{3}+(4\pi^{4}\!/15\!-\!8\pi^{2}\!/3)x^{5}$}\big{)}
                                                                                               \\[1.5 ex]
& \!\!\!\!\!-\!\!\!\!\! &\big{(}
\underbrace{\mbox{\small $(\pi^{4}/3)x^{3}+(4\pi^{4}\!/15\!-\!8\pi^{2}\!/3)x^{5}$}}_{(>0)}\big{)}\cos 2x.
\end{array}
\end{equation}
Then, according to the Lemmas 1.1. and 1.2. and the description of the method, the following inequalities:
$\sin y \!<\! \overline{T}_{k}^{\,\sin,0}(y) \,\, (k=9)$ and $\cos y \!<\! \overline{T}_{k}^{\,\cos,0}(y) \,\, (k=8)$
are true, for $y \in \big{(}0,\sqrt{(k+3)(k+4)}\,\big{)}$.

\smallskip
\noindent
For $x \in (0,1.136)$ it is valid:
\begin{equation}
\hspace*{-0.4cm}
\label{Ineq_68}
\begin{array}{rcl}
f(x)
& \!\!\!\!\!>\!\!\!\!\! &
\mbox{\small $x(\pi^{2}\!-\!4x^{2})^{2}$}\!\!-\!\!
\mbox{\small $\displaystyle\frac{(\pi^{2}\!-\!4x^{2})^{2}}{2}$}\overline{T}_{9}^{\,\sin,0}(2x)\!-\!\big{(}
\mbox{\small $(\pi^{4}/3)x^{3}+(4\pi^{4}\!/15\!-\!8\pi^{2}\!/3)x^{5}$}\big{)}
                                                                                               \\[1.5 ex]
& \!\!\!\!\!-\!\!\!\!\! &\big{(}
\underbrace{\mbox{\small $(\pi^{4}/3)x^{3}+(4\pi^{4}\!/15\!-\!8\pi^{2}\!/3)x^{5}$}}_{(>0)}\big{)}
\overline{T}_{8}^{\,\cos,0}(2x)
=P_{13}(x),
\end{array}
\end{equation}
where $P_{13}(x)$ is the polynomial
\begin{equation}
\label{Ineq_69}
\begin{array}{rcl}
P_{13}(x)
\!&\!\!=\!\!&\!
\mbox{\footnotesize $\displaystyle\frac{2x^{7}}{14175}$}
{\big (}
\mbox{\small $(-12\pi^{4}+120\pi^{2}-80)x^{6}\!-\!(-153\pi^{4}+1640\pi^{2}-1440)x^{4}$}
                                                                                  \\[2.0 ex]
\!&\!\!-\!\!&\!\!\mbox{\small $(1055\pi^{4}-11880\pi^{2}+15120)x^{2}\!+\!2295\pi^{4}
\!-\!30240\pi^{2}\!+\!75600$}{\big )}
                                                                                  \\[1.0 ex]
\!&\!\!=\!\!&\!
\mbox{\footnotesize $\displaystyle\frac{2x^{7}}{14175}$}
P_{6}(x).
\end{array}
\end{equation}
Then we determine the sign of the polynomial $P_{6}(x)$ for $x \in (0,1.136)$.
By introducing the substitute $z=x^{2}$, we get the third degree polynomial:
\begin{equation}
\label{Ineq_70}
\begin{array}{rcl}
P_{3}(z)
\!&\!\!=\!\!&\!
\mbox{\small $(-12\pi^{4}+120\pi^{2}-80)z^{3}\!-\!(-153\pi^{4}+1640\pi^{2}-1440)z^{2}$}
                                                                                  \\[2.0 ex]
\!&\!\!-\!\!&\!\!\mbox{\small $(1055\pi^{4}-11880\pi^{2}+15120)z\!+\!2295\pi^{4}
\!-\!30240\pi^{2}\!+\!75600$}.
\end{array}
\end{equation}

\noindent
A real numerical factorization of the polynomial $P_{3}(z)$, has been determined via {\sf Matlab}
software, and given with
\begin{equation}
\label{Ineq_71}
P_{3}(z)
\!=\!
\alpha(z\!-\!z_{1})(z^2\!+\!pz\!+\!q),
\end{equation}
where
$
\alpha\!=\!-64.556...,
z_{1}\!=\!1.290721...,
p
\!=\!-1.148...,
q
\!=\!8.365...;
$
whereby the inequality \mbox{$p^{\,2}\!-4q\!<\!0$} is true. The polynomial $P_{3}(z)$ has exactly
one simple real root with a symbolic radical representation and the corresponding numerical value
$z_{1}$. Let us notice that $\sqrt{z_1} = 1.136099... > 1.136$. Since $P_{3}(0)>0$ it follows
that $P_{3}(z)>0$ for $z \in (0, 1.136)$. Finally, we conclude that
\begin{equation}
\begin{array}{rcl}
\label{Ineq_72}
\hspace*{-0.4cm}
P_{6}(x) > 0 \,\,\mbox{\rm for}\,\, x \!\in\! (0,1.136)
& \!\!\!\!\!\Longrightarrow\!\!\!\!\! &
P_{13}(x) > 0 \,\,\mbox{\rm for}\,\, x \!\in\! (0,1.136)
                                                                                        \\[0.60 ex]
& \!\!\!\!\!\Longrightarrow\!\!\!\!\! &
f(x)> 0\,\,\mbox{\rm for}\,\, x \!\in\! (0,1.136).
\end{array}
\end{equation}
Let us notice that the least positive real root of the downward approximation of the function $f(x)$,
i.e. of the polynomial $P_{13}(x)$, is $x^{\ast} = \sqrt{z_1} = 1.136099...\;$.

\break

\noindent
\textbf{2)}
If $x \in [1.136,\pi/2)$, let us define the function
\begin{equation}
\begin{array}{rcl}
\label{Ineq_73}
\varphi(x)
\!&\!\!\!\!=\!\!\!\!&\!
f{\big(}\!\pi/2-x{\big)}\!=
\mbox{\small $-16x^{5}+40\pi x^{4}-32\pi^{2}x^{3}+8\pi^{3}x^{2}$}
                                                                                               \\[0.5 ex]
\!&\!\!\!\!-\!\!\!\!&\!
\mbox{\small $(16x^{4}-32\pi x^{3}+16\pi^{2}x^{2})$}\sin x\cos x
                                                                                                \\[0.5 ex]
\!&\!\!\!\!-\!\!\!\!&\!
\mbox{\small $(\pi^{2}/60)(\pi-2x)^{3}\big{(}(4\pi^{2}-40)x^{2}+(-4\pi^{3}+40\pi)x+\pi^{4}-5\pi^{2}\big{)}$}\sin^{2} x.
\end{array}
\end{equation}
Now we prove that $f(x)>0$ for $x \in [1.136,\pi/2)$, which is equivalent to $\varphi(x)>0$
for $x \in (\,0,c\,]$, where $c=\pi/2-1.136=\pi/2-142/125 \; (\,c=0.434\ldots\,)$. The function $\varphi(x)$
is also one concrete mixed trigonometric polynomial. According to the Theorem 1.5. the function
$\varphi(x)$ may be written in the following way:
\begin{equation}
\label{Ineq_74}
\hspace*{-2.0 mm}
\begin{array}{rcl}
\varphi(x)
& \!\!\!\!\!=\!\!\!\!\! &
\mbox{\small $-16x^{5}+40\pi x^{4}-32\pi^{2}x^{3}+8\pi^{3}x^{2}$}-
(\,\underbrace{\mbox{\small $8x^{4}-16\pi x^{3}+8\pi^{2}x^{2}$}}_{(>0)}\,)\sin2x
                                                                                                \\[0.5 ex]
\!&\!\!\!\!-\!\!\!\!&\!
\mbox{\small $(\pi^{2}/60)(\pi-2x)^{3}\big{(}(2\pi^{2}-20)x^{2}+(-2\pi^{3}+20\pi)x+\pi^{4}/2-5\pi^{2}/2\big{)}$}
                                                                                                \\[2.5 ex]
\!&\!\!\!\!+\!\!\!\!&\!
\mbox{\small $(\pi^{2}/60)(\pi-2x)^{3}$}\big{(}
\underbrace{\mbox{\small $(2\pi^{2}\!-\!20)x^{2}\!+\!(-\!2\pi^{3}\!+\!20\pi)x\!+\!\pi^{4}\!/2\!-\!5\pi^{2}\!/2$}}_{(>0)}\big{)}\!\cos 2x.
\end{array}
\end{equation}

\vspace*{-1.0 mm}

\noindent
Then, according to the Lemmas 1.1. and 1.2. and the description of the method, the following inequalities:
$\sin y \!<\! \overline{T}_{k}^{\,\sin,0}(y) \,\,(k=5)$ and $\cos y \!>\! \underline{T}_{k}^{\,\cos,0}(y) \,\,(k=6)$
are true, for $y \in \big{(}0,\sqrt{(k+3)(k+4)}\,\big{)}$.

\smallskip
\noindent
For $x \in (\,0,c\,]$ it is valid:
\begin{equation}
\hspace*{-0.4cm}
\label{Ineq_75}
\begin{array}{rcl}
\varphi(x)
& \!\!\!\!\!\!>\!\!\!\!\!\! &
\mbox{\small $-16x^{5}+40\pi x^{4}-32\pi^{2}x^{3}+8\pi^{3}x^{2}$}-
(\,\underbrace{\mbox{\small $8x^{4}-16\pi x^{3}+8\pi^{2}x^{2}$}}_{(>0)}\,)\overline{T}_{5}^{\,\sin,0}(2x)
                                                                                                \\[0.5 ex]
\!&\!\!\!\!\!-\!\!\!\!\!&\!
\mbox{\small $(\pi^{2}/60)(\pi-2x)^{3}\big{(}(2\pi^{2}-20)x^{2}+(-2\pi^{3}+20\pi)x+\pi^{4}/2-5\pi^{2}/2\big{)}$}
                                                                                               \\[2.5 ex]
\!&\!\!\!\!\!+\!\!\!\!\!&\!
\mbox{\small $(\pi^{2}\!/60)(\pi-2x)^{3}$}\!\big{(}
\underbrace{\mbox{\small $(2\pi^{2}\!-\!20)x^{2}\!+\!(-\!2\pi^{3}\!+\!20\pi)x\!+\!\pi^{4}\!/2\!-\!5\pi^{2}\!/2$}}_{(>0)}\big{)}\underline{T}_{6}^{\,\cos,0}(2x)
                                                                                               \\[0.0 ex]
\!&\!\!\!\!\!=\!\!\!\!\!&\!Q_{11}(x),
\end{array}
\end{equation}
where $Q_{11}(x)$ is the polynomial
\begin{equation}
\label{Ineq_76}
\begin{array}{rcl}
Q_{11}(x)
\!&\!\!=\!\!&\!
\mbox{\footnotesize $\displaystyle\frac{x^{2}}{2700}$}
{\big (}
\mbox{\small $(64\pi^{4}-640\pi^{2})x^{9}\!+\!(-160\pi^{5}+1600\pi^{3})x^{8}$}
                                                                                   \\[1.5 ex]
\!&\!\!+\!\!&\!\!\mbox{\small $(160\pi^{6}-2000\pi^{4}+4800\pi^{2}-5760)x^{7}$}
                                                                                   \\[0.5 ex]
\!&\!\!+\!\!&\!\!
\mbox{\small $(-80\pi^{7}+1880\pi^{5}-12000\pi^{3}+11520\pi)x^{6}$}
                                                                                   \\[0.5 ex]
\!&\!\!+\!\!&\!\!
\mbox{\small $(20\pi^{8}-1340\pi^{6}+12840\pi^{4}-20160\pi^{2}+28800)x^{5}$}
                                                                                   \\[0.5 ex]
\!&\!\!+\!\!&\!\!
\mbox{\small $(-2\pi^{9}+610\pi^{7}-8700\pi^{5}+36000\pi^{3}-57600\pi)x^{4}$}
                                                                                   \\[0.5 ex]
\!&\!\!+\!\!&\!\!
\mbox{\small $(-150\pi^{8}+4650\pi^{6}-34200\pi^{4}+28800\pi^{2}-86400)x^{3}$}
                                                                                   \\[0.5 ex]
\!&\!\!+\!\!&\!\!
\mbox{\small $(15\pi^{9}-1875\pi^{7}+15300\pi^{5}+194400\pi)x^{2}$}
                                                                                   \\[0.5 ex]
\!&\!\!+\!\!&\!\!
\mbox{\small $(450\pi^{8}-3150\pi^{6}-129600\pi^{2})x-45\pi^{9}+225\pi^{7}+21600\pi^{3}$}
{\big )}
                                                                                   \\[0.5 ex]
\!&\!\!=\!\!&\!\!
\mbox{\footnotesize $\displaystyle\frac{x^{2}}{2700}$}
Q_{9}(x).
\end{array}
\end{equation}
Then we determine the sign of the polynomial $Q_{9}(x)$ for $x \in (0,c]$.
Let us look at the fifth derivative of the polynomial $Q_{9}(x)$, as the fourth degree polynomial,
in the following form
\begin{equation}
\begin{array}{rcl}
\label{Ineq_77}
Q^{(5)}_{9}(x)
\!&\!\!\!\!=\!\!\!\!&\!
\mbox{\small $(967680\pi^{4}-9676800\pi^{2})x^{4}\!+\!(-1075200\pi^{5}+10752000\pi^{3})x^{3}$}
                                                                                           \\[2.0 ex]
\!&\!\!\!\!+\!\!\!\!&\!
\mbox{\small $(403200\pi^{6}-5040000\pi^{4}+12096000\pi^{2}-14515200)x^{2}$}
                                                                                           \\[2.0 ex]
\!&\!\!\!\!+\!\!\!\!&\!
\mbox{\small $(-57600\pi^{7}+1353600\pi^{5}-8640000\pi^{3}+8294400\pi)x$}
                                                                                           \\[2.0 ex]
\!&\!\!\!\!+\!\!\!\!&\!
\mbox{\small $2400\pi^{8}-160800\pi^{6}+1540800\pi^{4}-2419200\pi^{2}+3456000$}.
\end{array}
\end{equation}
A real numerical factorization of the polynomial $Q^{(5)}_{9}(x)$ has been determined via {\sf Matlab} software,
and given with
\begin{equation}
\label{Ineq_78}
Q^{(5)}_{9}(x)
\!=\!
\beta(x\!-\!x_{1})(x\!-\!x_{2})(x^2\!+\!px\!+\!q),
\end{equation}
where
$
\beta\!=\!-1245358.656...,
x_{1}\!=\!0.894...,
x_{2}\!=\!3.702...,
p
\!=\!1.106...,
q
\!=\!0.521...,
$
whereby the inequality \mbox{$p^{\,2}\!-4q\!<\!0$} is true.
The polynomial $Q^{(5)}_{9}(x)$ has exactly two simple real roots with a symbolic radical representation
and the corresponding numerical values $x_{1}$ and $x_{2}$. Therefore, the polynomial $Q^{(5)}_{9}(x)$ has no real roots for $x \in (\,0,c\,]$. Since $Q^{(5)}_{9}(0)<0$ it follows that $Q^{(5)}_{9}(x)<0$ for
$x \in (\,0,c\,]$. Furthermore, the polynomial $Q^{(4)}_{9}(x)$ is a monotonically decreasing function for $x \in (\,0,c\,]$ and $Q^{(4)}_{9}(c)>0$, so it follows that $Q^{(4)}_{9}(x)>0$ for $x \in (\,0,c\,]$.
Then, since the polynomial $Q^{'''}_{9}(x)$ is a monotonically increasing function for $x \in (\,0,c\,]$ and $Q^{'''}_{9}(c)<0$ it follows that $Q^{'''}_{9}(x)<0$ for $x \in (\,0,c\,]$. This implies that
the polynomial $Q^{''}_{9}(x)$ is a monotonically decreasing function for $x \in (\,0,c\,]$ and since $Q^{''}_{9}(c)>0$ it follows that $Q^{''}_{9}(x)>0$ for $x \in (\,0,c\,]$. Hence, the polynomial
$Q^{'}_{9}(x)$ is a monotonically increasing function for $x \in (\,0,c\,]$ and $Q^{'}_{9}(c)<0$, so it follows that $Q^{'}_{9}(x)<0$ for $x \in (\,0,c\,]$. Finally, since the polynomial $Q_{9}(x)$
is a monotonically decreasing function for $x \in (\,0,c\,]$ and $Q_{9}(c)>0$, we conclude that
\begin{equation}
\begin{array}{rcl}
\label{Ineq_79}
Q_{9}(x) \!>\! 0 \,\,\mbox{\rm for}\,\, x \!\in\! (\,0,c\,]
& \!\Longrightarrow\! &
Q_{11}(x) > 0 \,\,\mbox{\rm for}\,\, x \!\in\! (\,0,c\,]
                                                                                         \\[1.0 ex]
& \!\Longrightarrow\! &
\varphi(x)\!>\! 0\,\,\mbox{\rm for}\,\, x \!\in\! (\,0,c\,]
                                                                                         \\[1.0 ex]
& \!\Longrightarrow\! &
f(x)\!>\! 0\,\,\mbox{\rm for}\,\, x \!\in\! [1.136,\pi/2).
\end{array}
\end{equation}

\smallskip
\noindent
Let us notice that the least positive real root of the downward approximation of the function
$\varphi(x)$, i.e. of the polynomial $Q_{11}(x)$, is $x^{\ast}\!\!=\!0.630862\ldots \!>\!
c=0.434\ldots\;$. Elementary calculus gives that the constant $(8\pi^{4}/15-16\pi^{2}/3)$
is the best possible. The proof of the first inequality is completed.

\medskip
\noindent
\textbf{II}
We prove the inequality:
\begin{equation}
\label{Ineq_80}
x\sec^{2} \!x-\tan x<
\displaystyle\frac{(2\pi^{4}/3)x^{3}+(256/\pi^{2}-8\pi^{2}/3)x^{5}}{(\pi^{2}-4x^{2})^{2}}
\end{equation}
for $x \!\in\! (0,\pi/2)$. The requested inequality is equivalent to the inequality \mbox{$f(x)\!>\!0$} \hfill

\break

\noindent
for $x \!\in (0,\pi/2)$, where
\begin{equation}
\label{Ineq_81}
\begin{array}{rcl}
f(x)=\mbox{\small $-x(\pi^{2}\!-\!4x^{2})^{2}$}\!\!+\!\!\mbox{\small $(\pi^{2}\!-\!4x^{2})^{2}$}
\cos x\sin x\!+\!\big{(}
& \!\!\!\!\!\!(\!\!\!\!\!\! &\mbox{\small $2\pi^{4}/3)x^{3}$}+
                                                                                              \\[1.5 ex]
& \!\!\!\!\!\!(\!\!\!\!\!\! &\mbox{\small $256/\pi^{2}\!-\!8\pi^{2}\!/3$})\mbox{\small $x^{5}$}\big{)}\!\cos^{2}x
\end{array}
\end{equation}

\smallskip
\noindent
is one concrete mixed trigonometric polynomial.
Let us notice that $x=0$ is zero of the fifth order and $x=\pi/2$ is zero of the third order
of the function $f(x)$.

\smallskip
\noindent
Let us consider two cases:

\smallskip
\noindent
\textbf{1)}
If $x \in (0,0.858)$:

\smallskip
\noindent
According to the Theorem 1.5. the function $f(x)$ may be written in the following way:
\begin{equation}
\label{Ineq_82}
\begin{array}{rcl}
f(x)
& \!\!\!\!\!=\!\!\!\!\! &
\mbox{\small $-x(\pi^{2}\!-\!4x^{2})^{2}$}\!\!+\!\!
\mbox{\small $\displaystyle\frac{(\pi^{2}\!-\!4x^{2})^{2}}{2}$}\sin 2x\!+\!
\mbox{\small $(\pi^{4}/3)x^{3}+(128/\pi^{2}\!-\!4\pi^{2}\!/3)x^{5}$}
                                                                                               \\[1.5 ex]
& \!\!\!\!\!+\!\!\!\!\! &\big{(}
\underbrace{\mbox{\small $(\pi^{4}/3)x^{3}+(128/\pi^{2}\!-\!4\pi^{2}\!/3)x^{5}$}}_{(>0)}\big{)}\cos 2x.
\end{array}
\end{equation}

\smallskip
\noindent
Then, according to the Lemmas 1.1. and 1.2. and the description of the method, the inequalities:
$\sin y > \underline{T}_{k}^{\,\sin,0}(y) \,\,(k=7)$ and $\cos y > \underline{T}_{k}^{\,\cos,0}(y) \,\,(k=6)$
are true, for $y \in \big{(}0,\sqrt{(k+3)(k+4)}\,\big{)}$.

\smallskip
\noindent
For $x \in (0,0.858)$ it is valid:
\begin{equation}
\hspace*{-0.4cm}
\label{Ineq_83}
\begin{array}{rcl}
f(x)
& \!\!\!\!\!>\!\!\!\!\! &
\mbox{\small $-x(\pi^{2}\!-\!4x^{2})^{2}$}+
\mbox{\small $\displaystyle\frac{(\pi^{2}\!-\!4x^{2})^{2}}{2}$}\underline{T}_{7}^{\,\sin,0}(2x)\!+\!
\mbox{\small $(\pi^{4}/3)x^{3}\!+\!(128/\pi^{2}\!-\!4\pi^{2}\!/3)x^{5}$}
                                                                                               \\[1.5 ex]
& \!\!\!\!\!+\!\!\!\!\! &\big{(}
\underbrace{\mbox{\small $(\pi^{4}/3)x^{3}+(128/\pi^{2}\!-\!4\pi^{2}\!/3)x^{5}$}}_{(>0)}\big{)}
\underline{T}_{6}^{\,\cos,0}(2x)
=P_{11}(x),
\end{array}
\end{equation}
where $P_{11}(x)$ is the polynomial
\begin{equation}
\label{Ineq_84}
\hspace*{-2.0 mm}
\begin{array}{rcl}
P_{11}(x)
\!&\!\!\!=\!\!\!&\!
\mbox{\footnotesize $\displaystyle\frac{2x^{5}}{945\pi^{2}}$}
{\big (}
\mbox{\small $(56\pi^{4}\!-\!96\pi^{2}\!-\!5376)x^{6}\!+\!(-\!14\pi^{6}\!-\!372\pi^{4}\!+\!1008\pi^{2}\!+\!40320)x^{4}$}
                                                                                  \\[2.0 ex]
\!&\!\!\!+\!\!\!&\!\!\mbox{\small $(99\pi^{6}+756\pi^{4}-5040\pi^{2}-120960)x^{2}-252\pi^{6}+1260\pi^{4}+120960$}{\big )}
                                                                                  \\[1.0 ex]
\!&\!\!\!=\!\!\!&\!
\mbox{\footnotesize $\displaystyle\frac{2x^{5}}{945\pi^{2}}$}
P_{6}(x).
\end{array}
\end{equation}
Then we determine the sign of the polynomial $P_{6}(x)$ for $x \in (0,0.858)$.
By introducing the substitute $z=x^{2}$, we get the third degree polynomial:
\begin{equation}
\label{Ineq_85}
\begin{array}{rcl}
P_{3}(z)
\!&\!\!=\!\!&\!
\mbox{\small $(56\pi^{4}\!-\!96\pi^{2}\!-\!5376)z^{3}\!+\!(-\!14\pi^{6}\!-\!372\pi^{4}\!+\!1008\pi^{2}\!+\!40320)z^{2}$}
                                                                                  \\[2.0 ex]
\!&\!\!+\!\!&\!\!\mbox{\small $(99\pi^{6}+756\pi^{4}-5040\pi^{2}-120960)z-252\pi^{6}+1260\pi^{4}+120960$}.
\end{array}
\end{equation}

\noindent
A real numerical factorization of the polynomial $P_{3}(z)$, has been determined via {\sf Matlab}
software, and given with
\begin{equation}
\label{Ineq_86}
P_{3}(z)
\!=\!
\alpha(z\!-\!z_{1})(z^2\!+\!pz\!+\!q),
\end{equation}
where
$
\alpha\!=\!-868.572...,
z_{1}\!=\!0.737147...,
p
\!=\!0.077...,
q
\!=\!2.226...;
$
whereby the inequality \mbox{$p^{\,2}\!-4q\!<\!0$} is true. The polynomial $P_{3}(z)$ has exactly
one simple real root with a symbolic radical representation and the corresponding numerical value
$z_{1}$. Let us notice that $\sqrt{z_1} = 0.858573... > 0.858$. Since $P_{3}(0)>0$ it
follows that $P_{3}(z)>0$ for $z \in (0, 0.858)$. Finally, we conclude that
\begin{equation}
\begin{array}{rcl}
\label{Ineq_87}
\hspace*{-0.4cm}
P_{6}(x) > 0 \,\,\mbox{\rm for}\,\, x \!\in\! (0,\mbox{\small $0.858$})
& \!\!\!\!\!\Longrightarrow\!\!\!\!\! &
P_{11}(x) > 0 \,\,\mbox{\rm for}\,\, x \!\in\! (0,\mbox{\small $0.858$})
                                                                                         \\[1.0 ex]
& \!\!\!\!\!\Longrightarrow\!\!\!\!\! &
f(x)> 0\,\,\mbox{\rm for}\,\, x \!\in\! (0,\mbox{\small $0.858$}).
\end{array}
\end{equation}

\smallskip
\noindent
Let us notice that the least positive real root of the downward approximation of the function
$f(x)$, i.e. of the polynomial $P_{11}(x)$, is $x^{\ast}=\sqrt{z_{1}} = 0.858573...\;$.

\medskip
\noindent
\textbf{2)}
If $x \in [0.858,\pi/2)$, let us define the function
\begin{equation}
\begin{array}{rcl}
\label{Ineq_88}
\varphi(x)
\!&\!\!\!\!=\!\!\!\!&\!
f{\big(}\!\pi/2-x{\big)}\!=
\mbox{\small $16x^{5}-40\pi x^{4}+32\pi^{2}x^{3}-8\pi^{3}x^{2}$}
                                                                                               \\[1.5 ex]
\!&\!\!\!\!+\!\!\!\!&\!
\mbox{\small $(16x^{4}-32\pi x^{3}+16\pi^{2}x^{2})$}\sin x\cos x
                                                                                                \\[1.5 ex]
\!&\!\!\!\!+\!\!\!\!&\!
\mbox{\small $(1/(3\pi^{2}))(\pi-2x)^{3}\big{(}(96-\pi^{4})x^{2}+(\pi^{5}-96\pi)x+24\pi^{2}\big{)}$}\sin^{2} x.
\end{array}
\end{equation}
Now we prove that $f(x)>0$ for $x \in [0.858,\pi/2)$, which is equivalent to $\varphi(x)>0$
for $x \in (\,0, c\,]$, where $c = \pi/2-0.858 = \pi/2 - 429/500 \; (\,c = 0.712\ldots\,)$.
The function $\varphi(x)$ is also one concrete mixed trigonometric polynomial.
According to the Theorem 1.5. the function $\varphi(x)$ may be written in the following way:
\begin{equation}
\label{Ineq_89}
\begin{array}{rcl}
\varphi(x)
& \!\!\!\!\!=\!\!\!\!\! &
\mbox{\small $16x^{5}-40\pi x^{4}+32\pi^{2}x^{3}-8\pi^{3}x^{2}$}+
(\,\underbrace{\mbox{\small $8x^{4}-16\pi x^{3}+8\pi^{2}x^{2}$}}_{(>0)}\,)\sin2x
                                                                                                \\[2.5 ex]
\!&\!\!\!\!+\!\!\!\!&\!
\mbox{\small $(1/(6\pi^{2}))(\pi-2x)^{3}\big{(}(96-\pi^{4})x^{2}+(\pi^{5}-96\pi)x+24\pi^{2}\big{)}$}
                                                                                                \\[3.0 ex]
\!&\!\!\!\!-\!\!\!\!&\!
\mbox{\small $(1/(6\pi^{2}))(\pi-2x)^{3}$}\big{(}
\underbrace{\mbox{\small $(96\!-\!\pi^{4})x^{2}\!+\!(\pi^{5}\!-\!96\pi)x\!+\!24\pi^{2}$}}_{(>0)}\big{)}\!\cos 2x.
\end{array}
\end{equation}
Then, according to the Lemmas 1.1. and 1.2. and the description of the method, the inequalities:
$\sin y > \underline{T}_{k}^{\,\sin,0}(y) \,\,(k=7)$ and $\cos y < \overline{T}_{k}^{\,\cos,0}(y) \,\,(k=8)$
are true, for $y \in \big{(}0,\sqrt{(k+3)(k+4)}\,\big{)} $.

\smallskip
\noindent
For $x \in (\,0,c\,]$ it is valid:
\begin{equation}
\hspace*{-0.4cm}
\label{Ineq_90}
\begin{array}{rcl}
\varphi(x)
& \!\!\!\!\!\!>\!\!\!\!\!\! &
\mbox{\small $16x^{5}-40\pi x^{4}+32\pi^{2}x^{3}-8\pi^{3}x^{2}$}+
(\underbrace{\mbox{\small $8x^{4}-16\pi x^{3}+8\pi^{2}x^{2}$}}_{(>0)})\underline{T}_{7}^{\,\sin,0}(2x)
                                                                                                  \\[2.5 ex]
\!&\!\!\!\!\!+\!\!\!\!\!&\!
\mbox{\small $(1/(6\pi^{2}))(\pi-2x)^{3}\big{(}(96-\pi^{4})x^{2}+(\pi^{5}-96\pi)x+24\pi^{2}\big{)}$}
                                                                                                  \\[3.5 ex]
\!&\!\!\!\!\!-\!\!\!\!\!&\!
\mbox{\small $(1/(6\pi^{2}))(\pi-2x)^{3}$}\!\big{(}\!
\underbrace{\mbox{\small $(96\!-\!\pi^{4})x^{2}\!+\!(\pi^{5}\!-\!96\pi)x\!+\!24\pi^{2}$}}_{(>0)}\!\big{)}\overline{T}_{8}^{\,\cos,0}(2x)
                                                                                                  \\[3.5 ex]
\!&\!\!\!\!\!=\!\!\!\!\!&\!Q_{13}(x),
\end{array}
\end{equation}
where $Q_{13}(x)$ is the polynomial
\begin{equation}
\label{Ineq_91}
\begin{array}{rcl}
Q_{13}(x)
\!&\!\!=\!\!&\!
\mbox{\footnotesize $\displaystyle\frac{x^{3}}{945\pi^{2}}$}
{\big (}
\mbox{\small $(-8\pi^{4}+768)x^{10}\!+\!(20\pi^{5}-1920\pi)x^{9}$}
                                                                                  \\[2.0 ex]
\!&\!\!+\!\!&\!\!\mbox{\small $(-18\pi^{6}+112\pi^{4}+1728\pi^{2}-10752)x^{8}$}
                                                                                  \\[1.0 ex]
\!&\!\!+\!\!&\!\!
\mbox{\small $(7\pi^{7}-280\pi^{5}-576\pi^{3}+26880\pi)x^{7}$}

                                                                                  \\[1.0 ex]
\!&\!\!+\!\!&\!\!
\mbox{\small $(-\pi^{8}+252\pi^{6}-792\pi^{4}-24864\pi^{2}+80640)x^{6}$}
                                                                                  \\[1.0 ex]
\!&\!\!+\!\!&\!\!
\mbox{\small $(-98\pi^{7}+2076\pi^{5}+9408\pi^{3}-201600\pi)x^{5}$}
                                                                                  \\[1.0 ex]
\!&\!\!+\!\!&\!\!
\mbox{\small $(14\pi^{8}-1890\pi^{6}+1176\pi^{4}+191520\pi^{2}-241920)x^{4}$}
                                                                                  \\[1.0 ex]
\!&\!\!+\!\!&\!\!
\mbox{\small $(735\pi^{7}-5964\pi^{5}-80640\pi^{3}+604800\pi)x^{3}$}
                                                                                  \\[1.0 ex]
\!&\!\!+\!\!&\!\!
\mbox{\small $(-105\pi^{8}+5670\pi^{6}+15120\pi^{4}-574560\pi^{2})x^{2}$}
                                                                                  \\[1.0 ex]
\!&\!\!+\!\!&\!\!
\mbox{\small $(-2205\pi^{7}-2520\pi^{5}+234360\pi^{3})x+315\pi^{8}-30240\pi^{4}$}
{\big )}
                                                                                  \\[1.0 ex]
\!&\!\!=\!\!&\!\!
\mbox{\footnotesize $\displaystyle\frac{x^{3}}{945\pi^{2}}$}
Q_{10}(x).
\end{array}
\end{equation}
Then we determine the sign of the polynomial $Q_{10}(x)$ for $x \in (\,0,c\,]$.
Let us look at the sixth derivative of the polynomial $Q_{10}(x)$, as the fourth degree polynomial,
in the following form
\begin{equation}
\begin{array}{rcl}
\label{Ineq_92}
Q^{(6)}_{10}(x)
\!&\!\!\!\!=\!\!\!\!&\!
\mbox{\small $(-1209600\pi^{4}+116121600)x^{4}\!+\!(1209600\pi^{5}-116121600\pi)x^{3}$}
                                                                                           \\[2.0 ex]
\!&\!\!\!\!+\!\!\!\!&\!
\mbox{\small $(-362880\pi^{6}+2257920\pi^{4}+34836480\pi^{2}-216760320)x^{2}$}
                                                                                            \\[2.0 ex]
\!&\!\!\!\!+\!\!\!\!&\!
\mbox{\small $(35280\pi^{7}-1411200\pi^{5}-2903040\pi^{3}+135475200\pi)x$}
                                                                                           \\[2.0 ex]
\!&\!\!\!\!-\!\!\!\!&\!
\mbox{\small $720\pi^{8}+181440\pi^{6}-570240\pi^{4}-17902080\pi^{2}+58060800$}.
\end{array}
\end{equation}
A real numerical factorization of the polynomial $Q^{(6)}_{10}(x)$ has been determined via {\sf Matlab} software,
and given with
\begin{equation}
\label{Ineq_93}
Q^{(6)}_{10}(x)
\!=\!
\beta(x\!-\!x_{1})(x\!-\!x_{2})(x^2\!+\!px\!+\!q),
\end{equation}
where
$
\beta\!=\!-1704436.514...,
x_{1}\!=\!0.610...,
x_{2}\!=\!3.262...,
$
$
p
\!=\!0.731...,
$
$
q
\!=\!1.935...,
$
whereby the inequality \mbox{$p^{\,2}\!-4q\!<\!0$} is true.
The polynomial $Q^{(6)}_{10}(x)$ has exactly two simple real roots with a symbolic radical representation
and the corresponding numerical values $x_{1}$ and $x_{2}$. Since $Q^{(6)}_{10}(0)<0$, it follows that
$Q^{(6)}_{10}(x)<0$ for $x<x_{1}$ and since $Q^{(6)}_{10}(c)>0$, hence it follows that
$Q^{(6)}_{10}(x)>0$ for $x \in (x_{1},x_{2})$. Therefore, $Q^{(5)}_{10}(x)$ is a monotonically
decreasing function for $x<x_{1}$ and a monotonically increasing function for $x \in (x_{1},x_{2})$,
hence $Q^{(5)}_{10}(x)$ reaches the minimum at the point $x_{1}=0.610...$ in the interval
$(\,0,c\,]$. Then, since $Q^{(5)}_{10}(0)<0$ and $Q^{(5)}_{10}(c)<0$, it follows
that $Q^{(5)}_{10}(x)<0$ for $x \in (\,0,c\,]$. Furthermore, since the polynomial
$Q^{(4)}_{10}(x)$ is a monotonically decreasing function for $x \in (\,0,c\,]$ and
$Q^{(4)}_{10}(c)>0$ it follows that $Q^{(4)}_{10}(x)>0$ for $x \in (\,0,c\,]$.
This implies that the polynomial $Q^{'''}_{10}(x)$ is a monotonically increasing function for
$x \in (\,0,c\,]$ and $Q^{'''}_{10}(c)<0$, so it follows that $Q^{'''}_{10}(x)<0$
for $x \in (\,0,c\,]$. Hence, the polynomial $Q^{''}_{10}(x)$ is a monotonically decreasing
function for $x \in (\,0,c\,]$ and since $Q^{''}_{10}(c)>0$ it follows that
$Q^{''}_{10}(x)>0$ for $x \in (\,0,c\,]$. Then, since the polynomial $Q^{'}_{10}(x)$
is a monotonically increasing function for $x \in (\,0,c\,]$ and $Q^{'}_{10}(c)<0$
it follows that $Q^{'}_{10}(x)<0$ for $x \in (\,0,c\,]$. Finally, since the polynomial
$Q_{10}(x)$ is a monotonically decreasing function for $x \in (\,0,c\,]$ and
$Q_{10}(c)>0$, we conclude that
\begin{equation}
\begin{array}{rcl}
\label{Ineq_94}
Q_{10}(x) \!>\! 0 \,\,\mbox{\rm for}\,\, x \!\in\! (\,0,c\,]
& \!\Longrightarrow\! &
Q_{13}(x) > 0 \,\,\mbox{\rm for}\,\, x \!\in\! (\,0,c\,]
                                                                                        \\[0.60 ex]
& \!\Longrightarrow\! &
\varphi(x)\!>\! 0\,\,\mbox{\rm for}\,\, x \!\in\! (\,0,c\,]
                                                                                         \\[0.60 ex]
& \!\Longrightarrow\! &
f(x)\!>\! 0\,\,\mbox{\rm for}\,\, x \!\in\! [0.858,\pi/2).
\end{array}
\end{equation}

\smallskip
\noindent
Let us notice that the least positive real root of the downward approximation of the function
$\varphi(x)$, i.e. of the polynomial $Q_{13}(x)$, is $x^{\ast}=0.910490\ldots > c=0.712\ldots\;$.
Elementary calculus gives that the constant $(256/\pi^{2}-8\pi^{2}/3)$ is the best possible.
The proof of the second inequality is completed.
\end{proof}

\section{Conclusion} 

\noindent
\qquad$\!$
The previously described method may be applied to numerous \mbox{trigonometric} inequalities which correspond
to univariate mixed trigonometric polynomial functions. By using this method new results may be
obtained and the existing ones may be improved from the articles \mbox{\cite{Aharonov_Elias_2013}-\cite{Alirezaei_2013}}, \mbox{\cite{Bhayo_Sandor_2014}-\cite{Mortici_Debnath_Zhu_2015}},
\cite{Guo_Luo_Qi_2013b}-\cite{Jiang_Qi_2012}, \mbox{\cite{Klen_Visuri_Vuorinen_2010}},~\mbox{\cite{Mortici_2011}-\cite{Zhu_2006b}}
and the books \cite{Milovanovic_Rassias_2014}, \cite{Mitrinovic_1970}. Concrete results of the presented method
for proving some~inequalities, have been obtained in this article through the applications,
as well as in the articles \cite{Malesevic_Makragic_Banjac_2015}, \cite{Malesevic_Banjac_Jovovic_2015}
and \cite{Nenezic_Malesevic_Mortici_2015}.

\break

\end{document}